\documentclass[11pt]{amsart}

\usepackage[ansinew]{inputenc}
\usepackage[a4paper,dvips]{geometry}
\usepackage{latexsym}
\usepackage{amsfonts}
\usepackage{amsmath,amssymb}
\usepackage{mathtools}
\usepackage{bbm}
\usepackage[dvipsnames]{xcolor}
\usepackage{tikz}
\usetikzlibrary{decorations.pathreplacing,angles,quotes,calligraphy}

\usepackage{fixme}
\usepackage{graphicx}
\allowdisplaybreaks

\makeatletter
\newcommand{\addresseshere}{%
  \enddoc@text\let\enddoc@text\relax
}
\makeatother

\newtheorem{theorem}{Theorem}
\newtheorem{remark}[theorem]{Remark}
\newtheorem{lemma}[theorem]{Lemma}

\newtheorem{problem}{Problem}

\newcommand{\RR}{\mathbb{R}}

\newcommand{\Var}{{\mathbb V}\mathrm{ar}}
\newcommand{\E}{{\mathbb E}}
\newcommand{\cP}{{\mathcal P}}
\newcommand{\cS}{{\mathcal S}}
\newcommand{\cL}{{\mathcal L}}

\newcommand{\R}{{\mathbb R}}

\newcommand{\bx}{{\mathbf{x}}}

\newcommand{\bX}{{\mathbf{X}}}

\title[Partitions for stratified sampling]{\large Partitions for stratified sampling}

\author{Fran\c cois Cl\' ement}
\address{Sorbonne Universit\'e, CNRS, LIP6, Paris, France.}
\email{Francois.Clement@lip6.fr}

\author{Nathan Kirk}
\address{Queen's University Belfast, Belfast, United Kingdom.}
\email{nkirk09@qub.ac.uk}

\author{Florian Pausinger}
\address{Queen's University Belfast, Belfast, United Kingdom.}
\email{f.pausinger@qub.ac.uk}

\date{}

\begin{document}

\maketitle


\begin{abstract}
Classical jittered sampling partitions $[0,1]^d$ into $m^d$ cubes for a positive integer $m$ and randomly places a point inside each of them, providing a point set of size $N=m^d$ with small discrepancy.
The aim of this note is to provide a construction of partitions that works for arbitrary $N$ and improves  straight-forward constructions. We show how to construct equivolume partitions of the $d$-dimensional unit cube with hyperplanes that are orthogonal to the main diagonal of the cube. 
We investigate the discrepancy of such point sets and optimise the expected discrepancy numerically by relaxing the equivolume constraint using different black-box optimisation techniques.
\end{abstract}


\section{Introduction}
Classical jittered sampling provides $d$-dimensional point sets with small discrepancy in $[0,1]^d$ consisting of $N=m^d$ points for a positive integer $m\geq 2$.
The main idea is to partition the unit cube into $N$ axis-parallel cubes of volume $1/m^d$ and to pick a uniform random point from each cube. This construction avoids local clusters one usually obtains when sampling $N$ i.i.d uniform random points from $[0,1]^d$. 
The expected (star) discrepancy of a set $\cP_N$ of $N$ i.i.d.~uniform random points in $[0,1]^d$ is of order $\Theta(\sqrt{1/N})$; see \cite{hnww} for the first upper bound, \cite{aist} for the first upper bound with explicit constant and \cite{doerr1} for the first lower bound as well as \cite{gnewuch} for the current state of the art results in this context. 
Recently Doerr \cite{doerr} determined the precise asymptotic order of the expected star-discrepancy of a point set obtained from jittered sampling:
$$ \E \cL_{\infty} (\cP_N) = \Theta \left( \frac{\sqrt{d} \sqrt{1+\log(N/d)}}{N^{\frac{1}{2}+\frac{1}{2d} }} \right).$$ 
This result is our motivation to look for other constructions of partitions. In particular, we aim to find a construction that works for arbitrary $N$ while improving the expected (star) discrepancy of a set of $N$ i.i.d.~uniform random points.
In fact, for $1<p<\infty$, a stratified set derived from a partition into $N>2$ equivolume sets \emph{always} has a smaller expected $\cL_p$-discrepancy than a set consisting of $N$ i.i.d.~random points. 
	This \emph{strong partition principle} was proven in \cite[Theorem 1]{mf21} (see also \cite{stefan1} for a weaker form) and raises the question which partition yields the stratified sample with the smallest mean $\cL_p$-discrepancy -- if such a partition exists. Or, as a more modest goal, which construction works well in general?
	
\subsection{A construction.} There are several straightforward ways to partition the $d$-dimensional unit cube $[0,1]^d$ into $N$ sets of equal volume. 
We could for example partition the interval $[0,1]$ into $N$ subintervals of equal length and use this to define $N$ slices of equal volume. 
In fact, we can partition an arbitrary number of the $d$ generating unit intervals into subintervals of equal length to generate grid-type partitions of the $d$-dimensional unit cube. 
In every such example it is straightforward to characterise the sets in the partition (since they are axis-parallel rectangles) and to prove that these sets have indeed all the same volume. However, any such construction only works for a restricted number of points and not for arbitrary $N$.

The aim of this note is to characterise another family of equivolume partitions that was recently studied in the context of jittered sampling \cite{mf21}.
Given the unit square $[0,1]^2$ and $N-1$ parallel lines $H_{i}$ with $i=1, \ldots, N-1$, which are orthogonal to the main diagonal, $D$, of the square, we would like to arrange the lines such that we obtain an equivolume partition of the unit square; see Figure \ref{fig:simplePartition}.

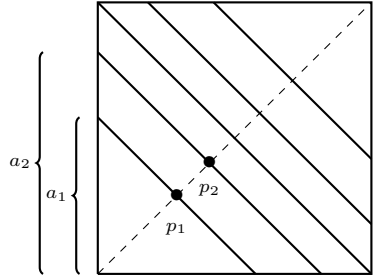
\begin{figure}[h]
\begin{tikzpicture}[scale=1.2]
		
		\draw[thick] (0,0) -- (3,0)-- (3,3) -- (0,3)-- (0,0);
		\draw[dashed] (0,0) -- (3,3);
		
		\draw[thick] (0,1.73) -- (1.73,0);
		\draw[thick] (0,2.45) -- (2.45,0);
		\draw[thick] (0,3) -- (3,0);
		\draw[thick] (3-1.73,3) -- (3,3-1.73);
		\draw[thick] (3-2.45,3) -- (3,3-2.45);
		
		\node at (0.865,0.865) {$\bullet$};
		\node at (0.865,0.5) {\tiny $p_1$};
		
		\node at (1.225,1.225) {$\bullet$};
		\node at (1.225,0.925) {\tiny $p_2$};
		
		\node at (-0.45, 0.865) {\tiny $a_1$};
		\node at (-0.85, 1.225) {\tiny $a_2$};
		
		\draw [decorate, decoration = {brace}, thick] (-0.2,0) --  (-0.2,1.73);
		\draw [decorate, decoration = {brace}, thick] (-0.6,0) --  (-0.6,2.45);
\end{tikzpicture}
\caption{Partition of the unit cube into $N=6$ equivolume slices that are orthogonal to the diagonal.} \label{fig:simplePartition}
\end{figure}

We denote the intersection $H_i \cap D$ of a line with the diagonal with $p_i$.
It is straightforward to calculate all points $p_i$ for arbitrary $N$. In fact, note that $H_i$ splits the unit square into two sets of volume $i/N$ and $1-i/N$. If $i\leq N/2$, we just need to look at the isosceles right triangle that $H_i$ forms with $(0,0)$. We know that this triangle has volume $i/N$. Denote the length of the two equal sides with $a_i>0$, then $a_i^2/2=i/N$ and $p_i = a_i/2$. Therefore, we get that
\begin{equation} \label{twoDim}
p_i = \sqrt{\frac{i}{2 N}},
\end{equation}
for all $i\leq N/2$. By symmetry, we also get the points $p_i$ with $i>N/2$.
The first question of this note is how to generalise this simple characterisation of equivolume partitions of the unit square to dimensions $d>2$?

To state this problem formally, we introduce further notation. Let $H_r^+$ be the positive half space defined as the set of all $\bx \in \RR^d$ satisfying
\begin{equation} \label{defH} x_1 + x_2 + \ldots + x_d  \geq r; \end{equation}
accordingly let  $H_r^-$ be the corresponding negative half space and let $H_r$ be the hyperplane of all points $\bx \in \RR^d$ satisying $x_1 + x_2 + \ldots + x_d = r$.
For a given $N$, we would like to find the points $0 < r_1 <\ldots < r_{N-1} < d$ such that the corresponding hyperplanes $H_{r_1}, \ldots, H_{r_{N-1}}$ define a partition of $[0,1]^d$ into $N$ equivolume sets. We call the set 
$$\cS(N,d):=\{r_i : i=1, \ldots, N-1\}$$ 
the \emph{generating set} of the partition. Using \eqref{twoDim} we see that
$$\cS(N,2) = \left\{ \sqrt{\frac{2i}{N}}: 1\leq i \leq N/2 \right\} \cup \left\{ 2 - \sqrt{\frac{2i}{N}}: 1 \leq i < N/2 \right\}.$$
Note that for even $N$ the point $\{1\}$ is in the set, while it is not for odd $N$.

\begin{problem} \label{p1}
For given dimension $d$ and number of sets $N$, characterise the set $\cS(N,d)$.
\end{problem}

\subsection{Optimisation}
Our construction is motivated by \cite[Example 2 and 3]{mf21}. It was shown that among all partitions of the unit square into two sets of equal-volume, the dividing hyperplane orthogonal to the diagonal and going through the center of the square gives the smallest discrepancy. Furthermore, it was observed that moving the dividing hyperplane along the diagonal and thus relaxing the equivolume constraint can further improve the expected discrepancy; see Figure \ref{fig:model}.
This motivates our second question.
For a given $N$, we would like to find the points $0 < r_1 < \ldots < r_{N-1} < d$ such that the corresponding hyperplanes $H_{r_1}, \ldots, H_{r_{N-1}}$ define a partition of $[0,1]^d$ into $N$ sets minimizing the expected discrepancy of the corresponding stratified sample. We call the set 
$$\cS^*(N,d):=\{r_i : i=1, \ldots, N-1\}$$ 
the \emph{minimal set} of the parameter pair $(N,d)$.

\begin{problem} \label{p2}
For given dimension $d$ and number of sets $N$, determine or approximate the minimal set $\cS^*(N,d)$.
\end{problem}

\begin{figure}[h!]
\centering
\begin{tikzpicture}[scale=2]
\draw [thick] (0,0) -- (1,0);
\draw [thick] (1,1) -- (1,0);
\draw [thick] (0,1) -- (1,1);
\draw [thick] (0,0) -- (0,1);
\draw [ultra thick] (0,1) -- (1,0);
\draw [ultra thick, dashed] (0.3,1) -- (1,0.3);
\draw [thick,->] (0.5,0.5)--(0.7,0.7);
\draw [thick,->] (0.5,0.5)--(0.3,0.3);
\draw [ultra thick, dashed] (0,0.7) -- (0.7,0);
\node at (0.5, 0.5) {$\bullet$};
\node at (0.5, -0.2) {\footnotesize 0.05};

\draw [thick] (2,0) -- (3,0);
\draw [thick] (3,1) -- (3,0);
\draw [thick] (2,1) -- (3,1);
\draw [thick] (2,0) -- (2,1);
\draw [ultra thick] (2.1,1) -- (3,0.1);
\node at (2.5, -0.2) {\footnotesize 0.0490};
\end{tikzpicture}
\caption{ Left: Best partition into two sets of equal volume and illustration of one-parameter family of partitions obtained from moving hyperplane along diagonal. Right: The partition of this family with the smallest expected discrepancy.}\label{fig:model} 
\end{figure}
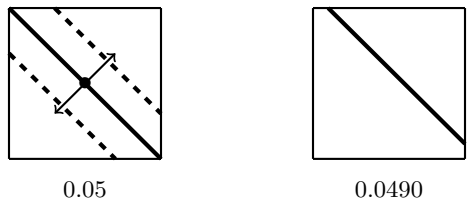

\subsection{Results.} \label{sec:results}
We fully answer Problem \ref{p1} providing two different characterisations for the set $\cS(N, d)$.
First, we use a result from \cite{cho} to derive an algebraic characterisation.

For $d \geq 2$, $0 \leq r \leq d$ and $1 \leq k \leq d$, we define $$f(k):= \frac{1}{d!}\sum_{j=0}^{k-1} (-1)^j \binom{d}{j} (k-j)^d$$ and set $f(0) = 0$.

\begin{theorem} \label{algebraicsol}
For arbitrary $r \in [0,d]$ and corresponding positive halfspace $H^+_r$, define $V_d^{+}: [0,d] \rightarrow [0,1]$ with
$$V_{d}^+(r) = \textnormal{vol} \left( [0,1]^d \cap H^+_r \right).$$ 
For arbitrary $V \in [0,1]$ let the integer $1\leq k \leq d$ be such that $f(k-1) \leq V \leq f(k)$.
The parameter $r$ of the positive half space $H_r^+$ such that $V_{d}^+(r)=V$ is a solution of the polynomial equation:
\begin{equation*}
		\frac{1}{d!} \sum_{j=0}^{k-1} (-1)^j \binom{d}{j} (d-j-r)^d = V_{d}^+(r).
\end{equation*}
\end{theorem}
Solving for $r_i$ with $V_{d}^+(r_i)=1-i/N$ and $1\leq i \leq N-1$ gives an algebraic characterisation of the set $\cS(N, d)$.
However, this description is not very practical and does not yield any insight into the structure of the set $\cS(N,d)$ as explained further in the subsequent sections. Therefore, we derive a second result which gives a probabilistic and more practical albeit approximate characterisation of $\cS(N, d)$. 

	\begin{theorem} \label{normalconverge}
	For $d \geq 2$, let $H_r^- = \{\mathbf{x} \in \mathbb{R}^d : x_1 + \ldots + x_d \leq r\}$ denote the negative half space with parameter $0 \leq r \leq d$ and let $V_{d}^-:[0,d] \rightarrow [0,1]$ with $V_{d}^-(r) = \textnormal{vol}\left( [0,1]^d \cap H_r^-  \right)$. If $\Phi$ is the cumulative distribution function of the standard normal distribution, then
	\begin{equation} \label{normaldistribution}
		V_{d}^-(r) \rightarrow \Phi \left( 2 \sqrt{3d} \left( \frac{r}{d} - \frac{1}{2} \right) \right)
	\end{equation}
	as $d \rightarrow \infty$ for all $r > 0$, i.e., $V_{d}^-$ converges pointwise to $\Phi$.
	\end{theorem}

To get the desired characterisation of the set $\mathcal{S}(N, d)$ containing the values $r_i$ such that $V_{d}^+(r_i) = \textnormal{vol} \left(H_{r_i}^+ \cap [0,1]^d \right) = 1-i/N$ for $1 \leq i \leq N-1$, we
apply the well known symmetry of the quantile function, 
\begin{equation} \label{phisym}
	\Phi^{-1}(x) = -\Phi^{-1}(1-x)
\end{equation} 
for $x \in [0,1]$.
Hence using the relation $V_{d}^-(r) = 1-V_{d}^+(r) = i/N$, we can deduce that for $1 \leq i \leq N-1$
\begin{equation*}
	r_i \approx \frac{d}{2} + \frac{\sqrt{d}}{2\sqrt{3}} \Phi^{-1}\left( \frac{i}{N} \right)
\end{equation*}
as $d \rightarrow \infty$.
Therefore with the above observations, Theorem \ref{normalconverge} gives the following probabilistic characterisation:
\begin{equation}\label{prob} 
\cS(N,d) \approx \left \{ \frac{d}{2} + \frac{\sqrt{d}}{2\sqrt{3}}\Phi^{-1}\left(\frac{i}{N} \right) : i=1, \ldots, N-1 \right\}.
\end{equation}

As a third result in order to tackle Problem \ref{p2}, we provide extensive numerical results on the set $\cS^{*}(N, d)$ in Section 6 for $d=2,3$. Using black-box optimisers, we are able to approximate the sets $\cS^{*}(N,2)$ and $\cS^{*}(N,3)$ for low values of $N$. We observe that, in two dimensions the equivolume stratification $\cS(N,2)$ improves the uniformly distributed random point sets by a factor of 2 for the $\mathcal{L}_2$-discrepancy, in line with results in \cite{mf21}, and approximations of $\cS^{*}(N,2)$ improve $\cS(N,2)$ by up to 8\% for the smallest values of $N$. In three dimensions, we see that approximations of $\cS^{*}(N,3)$ provide a slightly larger improvement of up to 10\% over the equivolume stratification.

\subsection{Outline.} Section \ref{sec:prelim} recalls important results and introduces further notation.
In Section \ref{sec:special} we illustrate the algebraic method in the special case of $d=3$, outlining the difficulties that occur already in very small dimensions. In Section \ref{sec:general}, we prove Theorem \ref{algebraicsol} while Theorem \ref{normalconverge} is proved in Section \ref{sec:normalDist}. The final section contains numerical results on the minimal sets $\cS^*(N,d)$.


\section{Preliminaries}
\label{sec:prelim}

It is a computationally very hard problem to calculate the volume of a convex polytope \cite{dyer} which inspired much research on approximation methods as well as exact calculations. In \cite{cho} the authors study the special case in which a polytope is obtained from a hypercube by clipping hyperplanes. The resulting formulas do not require heavy machinery but are technically complicated as much as they are intriguing.

For our note, the simple case of a hypercube clipped by only one hyperplane is of interest.
Our starting point is an interesting formula whose main idea can be traced back to a paper of Polya \cite{polya} and which seems to have appeared for the first time in \cite{barrow}. We refer to the introduction of \cite{cho} for more information.

\begin{theorem} {\cite[Theorem 1]{cho}} \label{thm:formula}
	Let $H^+_r$ be the halfspace defined by
	\begin{equation*} 
		\{\mathbf{x} \in \mathbb{R}^d | g(\bx):=\mathbf{a} \cdot \bx - r = a_1x_1 + a_2x_2 + \ldots + a_d x_d - r \geq 0 \},
	\end{equation*}
	with $\prod_{t=1}^d a_t \neq 0$. Then we have
	\begin{equation} \label{eqn1} 
		\textnormal{vol}\left( [0,1]^d \cap H^+_r \right) = \sum_{\mathbf{v} \in F^0 \cap H^+_r} \frac{(-1)^{ |\mathbf{v}_0|} g(\mathbf{v})^d}{d! \prod_{t=1}^d a_t}, 
	\end{equation}
	in which $F^0$ corresponds to the set of vertices of the cube, and $|\mathbf{v}_0|$ counts the number of zeros in the vector $\mathbf{v}$.
\end{theorem}

Our problem is in a way the inverse of the main problem studied in \cite{cho}. We know the volume of the intersection of a particular half space $H^+_r$ with the unit cube (and we know the normal vector of the hyperplane) and we are interested in the offset $r$. Thus, we can use Theorem \ref{thm:formula} to obtain an expression for the parameter $r$ given the volume of the intersection of the cube with the hyperplane.
To illustrate the theorem we look again at the two-dimensional case. For $i\leq N/2$ we have that $\textnormal{vol}([0,1]^2 \cap H^+_{r_i} )=1-i/N$. Moreover, we know that the three vertices $(1,0),(0,1)$ and $(1,1)$ of the unit square are contained in $[0,1]^2 \cap H^+_{r_i} $. Hence, we obtain
\begin{equation*}
1 - \frac{i}{N} = -\frac{(1-r_i)^2}{2} -\frac{(1-r_i)^2}{2} + \frac{(2 - r_i)^2}{2},
\end{equation*}
which can be simplified to
$$ r_i^2 = \frac{2i}{N}.$$
This coincides, up to the choice of the correct sign, with the result we already derived in the introduction. Note that even in the two-dimensional case we see that solving the problem algebraically requires us to pick our desired solution from a range of possibilities.


\section{The special case $d=3$} 
\label{sec:special}

To further illustrate the algebraic method, we explicitly solve the special case $d=3$. This case is still easy to visualise and already shows the difficulties we are facing beyond the two-dimensional case.
Let $[0,1]^3$ be the three-dimensional unit cube with vertices labelled $\mathbf{v}_1 = (0,0,0)$, $\mathbf{v}_2 = (1,0,0)$, $\mathbf{v}_3 = (0,1,0)$, $\mathbf{v}_4 = (0,0,1)$, $\mathbf{v}_5 = (1,1,0)$, $\mathbf{v}_6 = (1,0,1)$, $\mathbf{v}_7 = (0,1,1)$ and $\mathbf{v}_8 = (1,1,1)$. We ask for a characterisation of the set $\cS(N,3)$.
To be more precise, every point on a hyperplane $H_{r_i}$ satisfies
$$x_1 + x_2 + x_3 = r_i$$
and we are interested in the points $r_i$ such that $\textnormal{vol}([0,1]^3 \cap H^+_{r_i} )=1-i/N$ for $1\leq i \leq N-1$.

The following lemma characterises the points $r_i$ as roots of polynomials of degree 3.

\begin{lemma} \label{lem:r32}
For arbitrary $r \in [0,3]$ and corresponding positive halfspace $H^+_r$, define $V_{3}^+: [0,3]\rightarrow [0,1]$ with
$$V_{3}^+(r) := \textnormal{vol} \left( [0,1]^3 \cap H^+_r \right).$$ 
Fixing a value $V \in [0,1]$, the parameter value $r$ such that $V_{3}^+(r) = V$
can be obtained as a solution of the corresponding polynomial equation:
	\begin{equation*}
		0 = \begin{cases}
			(3-r)^3 - 6V & \textnormal{if } 0 \leq V< \frac{1}{6}; \\
			(3-r)^3 - 3(2-r)^3 - 6V & \textnormal{if } \frac{1}{6} \leq V < \frac{5}{6}; \\
			(3-r)^3 - 3(2-r)^3 + 3(1-r)^3 - 6V & \textnormal{if } \frac{5}{6} \leq V < 1.
		\end{cases}
	\end{equation*}
\end{lemma}

\begin{proof} 
We use Theorem \ref{thm:formula} and consider the three cases separately. We must first remark several important facts that will be useful in this derivation. It can be seen with some small calculations (via Theorem \ref{thm:formula}) that when $r=1$ and $r=2$, the volume of intersection of $[0,1]^3 \cap H_r^+$ is $5/6$ and $1/6$ respectively. Hence, for example, the range of parameter $0 <r \leq 1$ corresponds to a range of volume of intersection $\frac{5}{6} \leq V <1$.

In the first instance, let $0 \leq V < \frac{1}{6}$ and we see that $F^0 \cap H_r^+ = \{\mathbf{v}_8\}$. Therefore, $g(\mathbf{v}_8) = (1,1,1)\cdot(1,1,1)-r = 3-r$, and we obtain
\begin{equation*}
		V = \frac{(-1)^0 (3-r)^3}{3! \cdot 1}
	\end{equation*}
	which can be rewritten as
	\begin{equation*}
		0 = (3-r)^3 - 6V.
	\end{equation*}
 For the second case, if $\frac{1}{6} \leq V \leq \frac{5}{6}$, then $F^0 \cap H_r^+ = \{\mathbf{v}_5, \mathbf{v}_6, \mathbf{v}_7,\mathbf{v}_8\}$. Since the vectors $\mathbf{v}_5, \mathbf{v}_6$ and $\mathbf{v}_7$ all contain one zero, we have that $g(\mathbf{v}_i) = 2 - r$ for $i = 5, 6, 7.$ Hence we get, 
	\begin{equation*}
		V = \frac{(-1)^0 (3-r)^3}{3! \cdot 1} + 3 \frac{(-1)^1 (2-r)^3}{3! \cdot 1}
	\end{equation*}
or more simply
	\begin{equation*}
		0 = (3-r)^3 - 3(2-r)^3 - 6V.
	\end{equation*}
Finally, for $\frac{5}{6} \leq V < 1$ we have all the vertices in the intersection $F^0 \cap H_r^+$ except $\mathbf{v}_1$. Note that the vectors $\mathbf{v}_2, \mathbf{v}_3$ and $\mathbf{v}_4$ all contain two zero components hence $g(\mathbf{v}_i) = 1-r$ for $i=2, 3, 4$. Therefore,
	\begin{equation*}
			V = \frac{(-1)^0 (3-r)^3}{3! \cdot 1} + 3 \frac{(-1)^1 (2-r)^3}{3! \cdot 1} + 3 \frac{(-1)^2 (1-r)^3}{3! \cdot 1}
	\end{equation*}
	which can be re written as, 
	\begin{equation*}
		0 = (3-r)^3 - 3(2-r)^3 + 3(1-r)^3 - 6V
	\end{equation*}
	as required.
\end{proof}

Using the last Lemma, we obtain the following full characterisation of $r$ for a given $V$:
\begin{equation*}
r= \begin{cases}
3-\sqrt[3]{6V}  & \textnormal{if } 0 \leq V < \frac{1}{6}; \\
\frac{3}{2}+\frac{1}{4} \sqrt[3]{3} \left(1+i \sqrt{3}\right) \sqrt[3]{C+4 V -2}+\frac{3^{2/3} \left(1-i \sqrt{3}\right)}{4\sqrt[3]{C+4 V -2}}		 & \textnormal{if } \frac{1}{6} \leq V< \frac{5}{6}; \\
\sqrt[3]{6-6V}			 & \textnormal{if } \frac{5}{6} \leq V < 1,
\end{cases}
\end{equation*}
in which $C=\sqrt{16 V^2-16 V +1}$.
The resulting function is illustrated in Figure \ref{fig:plot}.

\begin{figure}
\includegraphics[width=9cm]{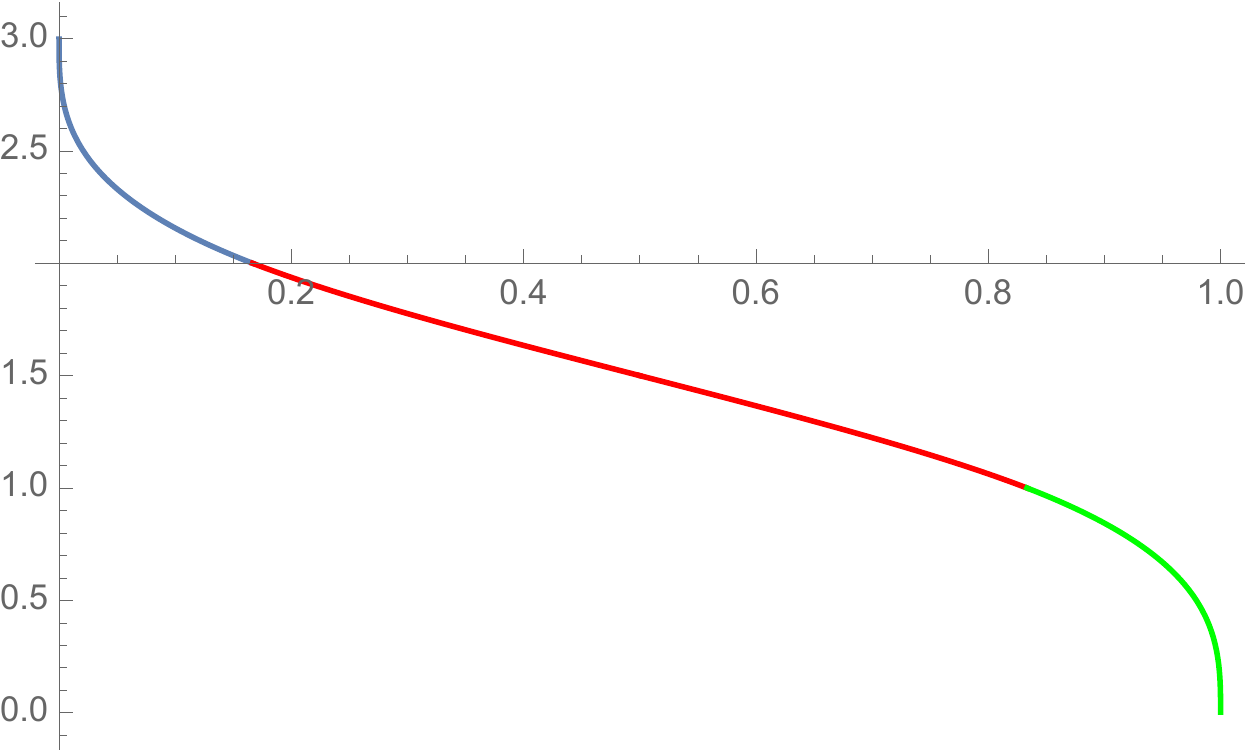}
\caption{ The $x$-axis contains the values of $V$, while the $y$-axis shows the corresponding $r$ for a given volume $V$. The three different colors illustrate the three different cases. \label{fig:plot}}
\end{figure}


\section{The general case} 
\label{sec:general}

We now proceed to the general $d-$dimensional case while emphasising that one of the key ingredients of the formula in Theorem \ref{thm:formula} is to determine which vertices lie inside the half space and subsequently counting how many of these vertices contain $k$ zeros, with $1 \leq k \leq d$. In the context of our problem, i.e. when considering hyperplanes orthogonal to the main diagonal, this can be done in a simple manner that is summarised below in Lemma \ref{lem:gen1}. 

\begin{lemma} \label{lem:gen1}
For $d \geq 2$, let $H_r^+$ be the positive half space for $r$ with $0 < r \leq d$.
Let $F^0$ denote the set of vertices of the $d-$dimensional hypercube $[0,1]^d$. 
For an integer $1\leq k \leq d$, let $\mathcal{V}_k := \{\mathbf{v} \in F^0 : \left| \mathbf{v}_0 \right| < k\}$ where $|\mathbf{v}_0|$ counts the number of zeros in a vector $\mathbf{v}$. 
Then, for every $k \leq \lceil d-r \rceil$ we have that
\begin{equation*}
		\mathcal{V}_k \subset H_r^+.
\end{equation*}
Moreover, $\left| \mathcal{V}_k \setminus \mathcal{V}_{k-1} \right| =  \binom{d}{k-1}.$
\end{lemma}

\begin{proof}
Let $1 \leq k \leq d$. By definition, a vertex $\mathbf{v}=(v_1, \ldots, v_d)$ is in $\mathcal{V}_k$ if and only if $\mathbf{v}$ contains at most $k-1$ zeros. Hence this vertex contains at least $d-(k-1)$ components equal to one. Therefore given $k$ such that $k \leq \lceil d-r \rceil$ and a vector $\mathbf{v} \in \mathcal{V}_k$,
\begin{eqnarray}
v_1 + v_2 + \cdots + v_d - r &\geq& d - (k-1) - r \nonumber \\ &\geq& d-(k-1) - (d - k + 1) = 0 \nonumber
\end{eqnarray}
since $k \leq \lceil d-r \rceil \Rightarrow r \leq d - k + 1$. That is, $\mathbf{v} \in H_r^+$ and hence $\mathcal{V}_k \subset H_r^+$ as required.
Finally, there are ${d \choose k-1}$ vertices containing exactly $k-1$ ones and hence by definition these are precisely the elements in the difference $\mathcal{V}_k \setminus \mathcal{V}_{k-1}$.
\end{proof}

Using the previous result, we now give a closed form expression for the volume of the intersection of $[0,1]^d$ and a positive halfspace $H_r^+$ for an integer $r$ with $0 \leq r \leq d$ which will be used to calculate the critical values of the volume of intersection, $V^+_d(r)$, determining the form of our algebraic equation in variable $r$.

\begin{lemma} \label{lem:gen2}
For $d \geq 2$, let $r = d-k$ be an integer with $1 \leq k \leq d$ and let $H_r^+$ be the corresponding positive half space. Define $V_d^+: [0,d] \rightarrow [0, 1]$ with $$V_d^+(r) \coloneqq \textnormal{vol} \left( [0, 1]^d \cap H_r^+ \right).$$ Then
$$V_d^+(r) = \frac{1}{d!}\sum_{j=0}^{k-1} (-1)^j \binom{d}{j} \left( k-j  \right)^d.$$
\end{lemma}

\begin{proof}
By definition, every vertex $\mathbf{v}$ of $[0,1]^d$ has $d - \left| \mathbf{v}_0 \right|$ components equal to one with the rest equalling zero. Hence from the formula in Theorem \ref{thm:formula}, with $g(\mathbf{v}) = d - \left| \mathbf{v}_0 \right| - r$
\begin{equation} \label{eqn:ver2}
		\textnormal{vol} \left( [0,1]^d \cap H^+_r \right) = \sum_{\mathbf{v} \in F^0 \cap H_d^+} \frac{(-1)^{|\mathbf{v}_0|} (d- \left| \mathbf{v}_0 \right| - r)^d}{d!}.
\end{equation}
Now given $0 \leq k \leq d$, let $r=d-k$ which implies $\mathcal{V}_k \subset H_r^+$, from Lemma \ref{lem:gen1}. Moreover, we observe that for every $i \leq k$, $\mathbf{v} \in \mathcal{V}_i \setminus \mathcal{V}_{i-1}$ implies $\left| \mathbf{v}_0 \right| = i-1$ and $\left| \mathcal{V}_i \setminus \mathcal{V}_{i-1} \right| = \binom{d}{i-1}$. Therefore we obtain a closed formula for the volume of intersection for each integer $0 \leq r \leq d$ as follows
\begin{eqnarray}
\textnormal{vol} \left( [0,1]^d \cap H^+_r\right) &=& \sum_{\mathbf{v} \in F^0 \cap H_r^+} \frac{(-1)^{|\mathbf{v}_0|} (d- \left| \mathbf{v}_0 \right| -r)^d}{d!} \nonumber  \\ 
&=& \sum_{i=1}^k \frac{ (-1)^{i-1} \binom{d}{i-1} \left( d-(i-1)-r \right)^d}{d!}  \nonumber \\ 
&=& \sum_{i=1}^k \frac{(-1)^{i-1} \binom{d}{i-1} \left( d-(i-1)-(d-k) \right)^d}{d!} \nonumber \\ 
&=& \sum_{i=1}^k \frac{(-1)^{i-1} \binom{d}{i-1} \left( k-(i-1)  \right)^d}{d!} \nonumber \\ 
&=& \sum_{j=0}^{k-1} \frac{(-1)^j \binom{d}{j} \left( k-j  \right)^d}{d!} \nonumber
	\end{eqnarray}
\end{proof}

With this, we can finally prove Theorem \ref{algebraicsol}.

\begin{proof}[Proof of Theorem \ref{algebraicsol}]
For a given fixed volume $V \in [0, 1]$, let $1 \leq k \leq d$ be such that $f(k-1) \leq V \leq f(k)$. Let $r$ be such that $V_d^+(r)=V$, then we know from Lemma \ref{lem:gen1} and \ref{lem:gen2} that $\mathcal{V}_k \subset H_r^+$. Applying a similar argument as in Lemma \ref{lem:gen2} we have for every $i \leq k, \mathbf{v} \in \mathcal{V}_i \setminus \mathcal{V}_{i-1}$ implies $|\mathbf{v}_0|=i-1$ and $|\mathcal{V}_i \setminus \mathcal{V}_{i-1}|= \binom{d}{i-1}$ recalling that $|\mathbf{v}_0|$ counts the number of zeros in vector $\mathbf{v}$.
Hence using the formula from Theorem \ref{thm:formula},
\begin{eqnarray}
	V = \textnormal{vol} \left( [0,1]^d \cap H^+_r\right) &=& \sum_{\mathbf{v} \in F^0 \cap H_r^+} \frac{(-1)^{|\mathbf{v}_0|} (d- \left| \mathbf{v}_0 \right| -r)^d}{d!} \nonumber  \\ 
	&=& \sum_{i=1}^k \frac{(-1)^{i-1} \binom{d}{i-1} \left( d-(i-1)-r \right)^d}{d!}. \nonumber
\end{eqnarray}
This can be rearranged to yield
\begin{equation} \label{polyeqn}
		\frac{1}{d!} \sum_{j=0}^{k-1} (-1)^j \binom{d}{j} (d-j-r)^d = V.
\end{equation}
\end{proof}

\begin{remark} \label{rem:comments}

Commenting of the work completed so far, we can now provide a full characterisation of the generating set $\cS(N,d)$ for arbitrary $N$ and $d$ via the solutions to order d polynomial equations of the form \ref{polyeqn}. It is worth noting that the Abel-Ruffini Theorem informs us that given a general polynomial of order (or in our case, dimension) greater than $4$, we are not guaranteed to obtain an exact closed form solution. However, we can of course solve for variable $r$ numerically to arbitrary degree of accuracy using for example Newton's method. As a second remark, the algebraic characterisation unfortunately does not give insight into any underlying structure that the parameter values might hold. Due to these drawbacks, in the next section we derive a second result from which a probabilistic approximation of $\cS(N, d)$ can be stated as in Theorem \ref{normalconverge}.
\end{remark}


\section{Approximation by normal distribution} 
\label{sec:normalDist}

\subsection{Probabilistic characterisation}
The formula in Theorem \ref{algebraicsol} is a precise characterisation of $r$ and lets us plot the graph of $V_{d}^+(r)$ for increasing $r$; i.e. knowing that $r \in [0,d]$ we can explicitly calculate the volume $V_{d}^+(r)$ of the intersection $[0,1]^d \cap H_r^+$.
As pointed out in Remark \ref{rem:comments}, the inverse calculation is, however, limited by fundamental algebraic obstacles.

This motivates the search for a different route to a neat characterisation of the sets $\cS(N,d)$ for arbitrary parameters $d$ and $N$. We start from the explicit solutions that we have for $d=3$; see Figure \ref{fig:plot}. This plot looks very similar to the quantile function of a normal distribution. In fact, if we interpret the set $\mathcal{S}(N,3)$ as a point sample in $[0,3]$, some numerical exploration leads us to conjecture Theorem \ref{normalconverge} which we prove in the following.

\begin{proof}[Proof of Theorem \ref{normalconverge}]
Let $\bX=(X_1,\ldots,X_d)$ be a uniform random point in $[0,1]^d$ and let $P$ be its projection onto the diagonal, spanned by the unit vector $v=(1/\sqrt{d})\mathbf{1}$, where $\mathbf{1}=(1,\ldots,1)^\top$, i.e.
	\[
	P=(\bX\cdot v) v=\left(\frac1d \sum_{i=1}^d X_i\right) \mathbf{1}. 
	\]
Let's consider the length $\|P\|=\sqrt d\cdot \frac1d \sum_{i=1}^d X_i$. As $\|P\|$ has support on $[0,\sqrt d]$ growing with $\sqrt d$ we re-scale and consider 
	\[
	R=\frac1{\sqrt d}\|P\|=\frac1d \sum_{i=1}^d X_i. 
	\]
The value $R$ is required in order to allow the implementation of the central limit theorem later, however note that we are actually interested in the value
	\[
	R'=\sqrt{d}\|P\|.
	\]
The interpretation of the cumulative distribution function $F_R$ of $R$ is as follows: 
	\[
	F_R(t)=\mathbb{P}(R\le t)=\mathbb{P}(\|P\|\le \sqrt d t)=\mathbb{P}(R' \leq dt)=\mathbb{P}(\bX\in H_{dt}^-)=V^-_d (dt), 
	\]
	where $0 \leq t \leq 1$, and $H_r^-=\{\bx\in \R^d: x_1 + \cdots + x_d \leq r \}$ is the negative half space at position $r$ with 
	$V^-_d : [0,d] \rightarrow [0,1]$ denoting the volume of $H_r^- \cap [0,1]^d$. In other words, 
	\begin{equation}
		V^-_d (r)=F_R\Big(\frac{r}{d}\Big)
		\label{eq1a}
	\end{equation}
	
	The value $R$ is the average of the variables $X_1,\ldots,X_d$ which are i.i.d. uniform in $[0,1]$ with mean $\mu=\E X_1=1/2$  and variance $\sigma^2=\Var(X_1)=1/12$. The central limit theorem states that the normalized variables 
	\[
	Z_d= \sqrt{d} \frac{R-\mu}{\sigma}=2\sqrt{3d} \Big(R-\frac12\Big)
	\]
	converge in distribution to a standard normal as $d\to \infty$. One way to express this convergence is in terms of cumulative distribution functions (CDFs). If $\Phi$ is the CDF of a standard normal, we have 
	\begin{equation}
		F_{Z_d}(t)\to \Phi(t),
		\label{eq2a}
	\end{equation}
	as $d\to \infty$, for all $t\in \R$ (i.e. $F_{Z_d}$ converges to $\Phi$ pointwise). Combining \eqref{eq1a} and \eqref{eq2a} we obtain 
	\begin{align*}
		V^-_d (r) =F_R\Big(\frac{r}{d}\Big)=F_{Z_d}
		\Bigg(2\sqrt{3 d} \bigg(\frac{r}{d}-\frac12\bigg) \Bigg)
		\rightarrow	
		\Phi\Bigg({2\sqrt 3d}\bigg(\frac{r}{d}-\frac{1}{2}\bigg)\Bigg)
	\end{align*}
	pointwise for all $r>0$ as $d \rightarrow \infty$.
\end{proof}

As outlined in the introduction, we can now derive our probabilistic characterisation via the symmetry of the quantile function.
The rate of the convergence described by the Central Limit Theorem is customarily measured via the Berry-Esseen Theorem \cite{berry}; see Appendix \ref{app1}

\subsection{Numerical experiments} As a first numerical experiment, we investigate the convergence in Theorem \ref{normalconverge}. For $d=3, 5 \hspace{1mm} \textnormal{and} \hspace{1mm} 10$ and $N=10000$, we numerically calculate parameter values $r_i$ $(1 \leq i \leq 9999)$ from Theorem \ref{algebraicsol} and subsequently use these values to calculate probabilities from the cumulative distribution function in \eqref{normaldistribution}. Comparing these probability values at $r_i$ with the desired cumulative volume of intersection, i.e. $i/10000$ for $1 \leq i \leq 10000$, the error is plotted in Figure \ref{fig:convplot} for increasing dimension. As expected, for increasing dimension the error becomes on average smaller, showing the desired convergence. Note that the values we obtain in our experiments are much smaller than what is guaranteed by the general bound in \eqref{rateOfConv}.

\begin{figure}[h!]
	\includegraphics[width=7cm]{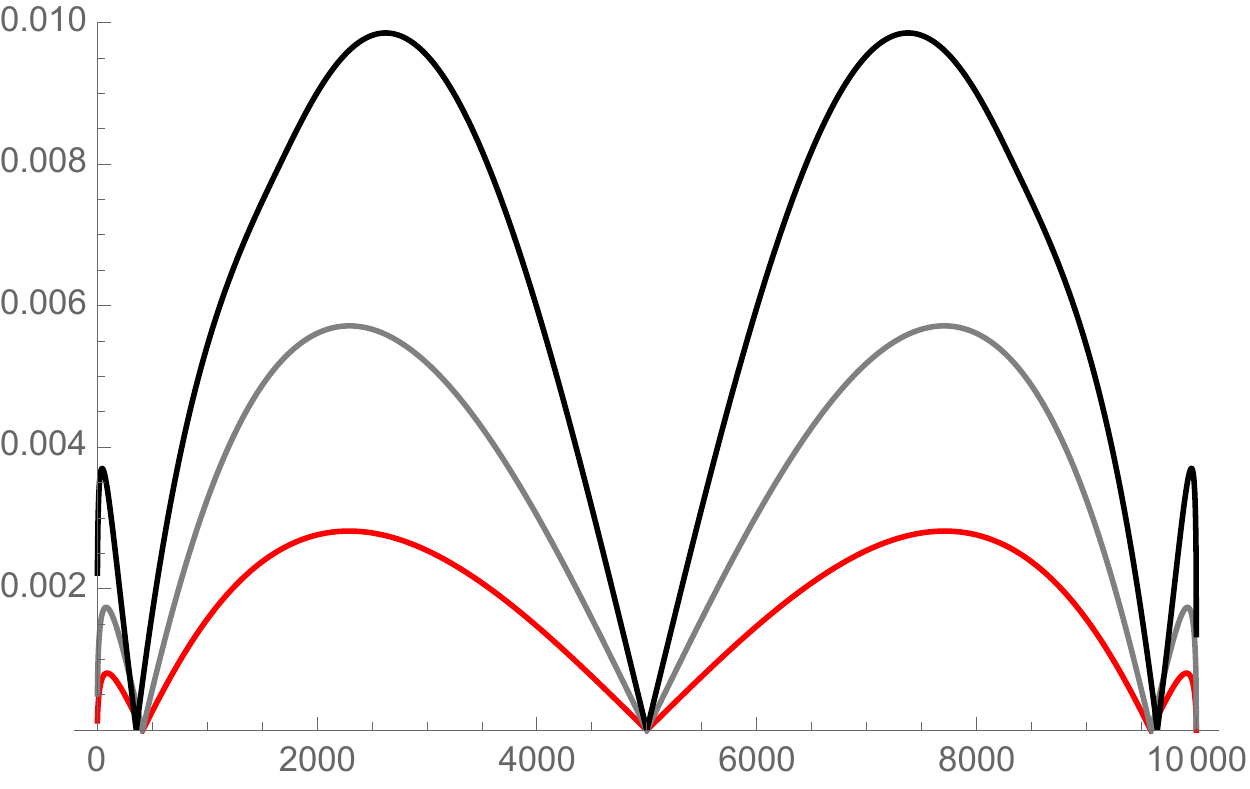}
	\caption{Illustration of convergence to the cumulative distribution function as proved in Theorem \ref{normalconverge}. The $x$ axis represents the index, $i$, of the set in the partition with $1\leq i \leq 10000$, while the $y$-axis contains the absolute difference $\left|V_{d}^-(r_i) - \Phi \left( 2 \sqrt{3d} \left( \frac{r_i}{d} - \frac{1}{2} \right) \right)\right |$. Black: $d=3$, Gray: $d=5$, Red: $d=10$.} \label{fig:convplot}
\end{figure}

Next, we would like to know how close the partitions generated by the probabilistic characterisation of $\cS(N, d)$ in \eqref{prob} are to being equivolume. 
We fix $d=5$ and set $N=100, 1000$ and $10000$. Using two consecutive parameter values from the normal approximation \eqref{prob}, the volume of a single partitioning set, generated by the probabilistic values for the parameter $r$, can be calculated using \eqref{polyeqn}. The resulting approximate volume for each partitioning set is then compared with the desired volume of $1/N$; see Figure \ref{fig:volplot}. We observe that the error is bounded inside the main body of the cube but much larger in the extreme ends which indicates that our deterministic distribution is not well approximated by the tails of the normal distribution. However, the number of sets in the partition that are affected from this large deviation is relatively small.
Our numerical results indicate that the relative error is at most about $7\%$ for \emph{most} of the sets; i.e. if a set is expected to have volume $10^{-k}$, the volume of the approximated set is in the interval $10^{-k}\pm 7 \cdot 10^{k-2}$.
To be more precise, note for example that in the case $d=5$ and $N=10000$ we leave the initial segment of the cube, i.e. for parameter values $-1 \leq r \leq 0$, at partition set number $84 = \big\lceil \frac{10000}{5!} \big\rceil$ and symmetrically return to the final segment at set number $10000-84 = 9916$ for parameter ranging $-d \leq r \leq -(d-1)$. 
These values are denoted in Figure \ref{fig:volplot} by two red dashes. Hence, $9832$ of all $10000$ sets, or about $98\%$ of the sets, are well approximated.

\begin{figure}[h!]
	\includegraphics[width=5cm]{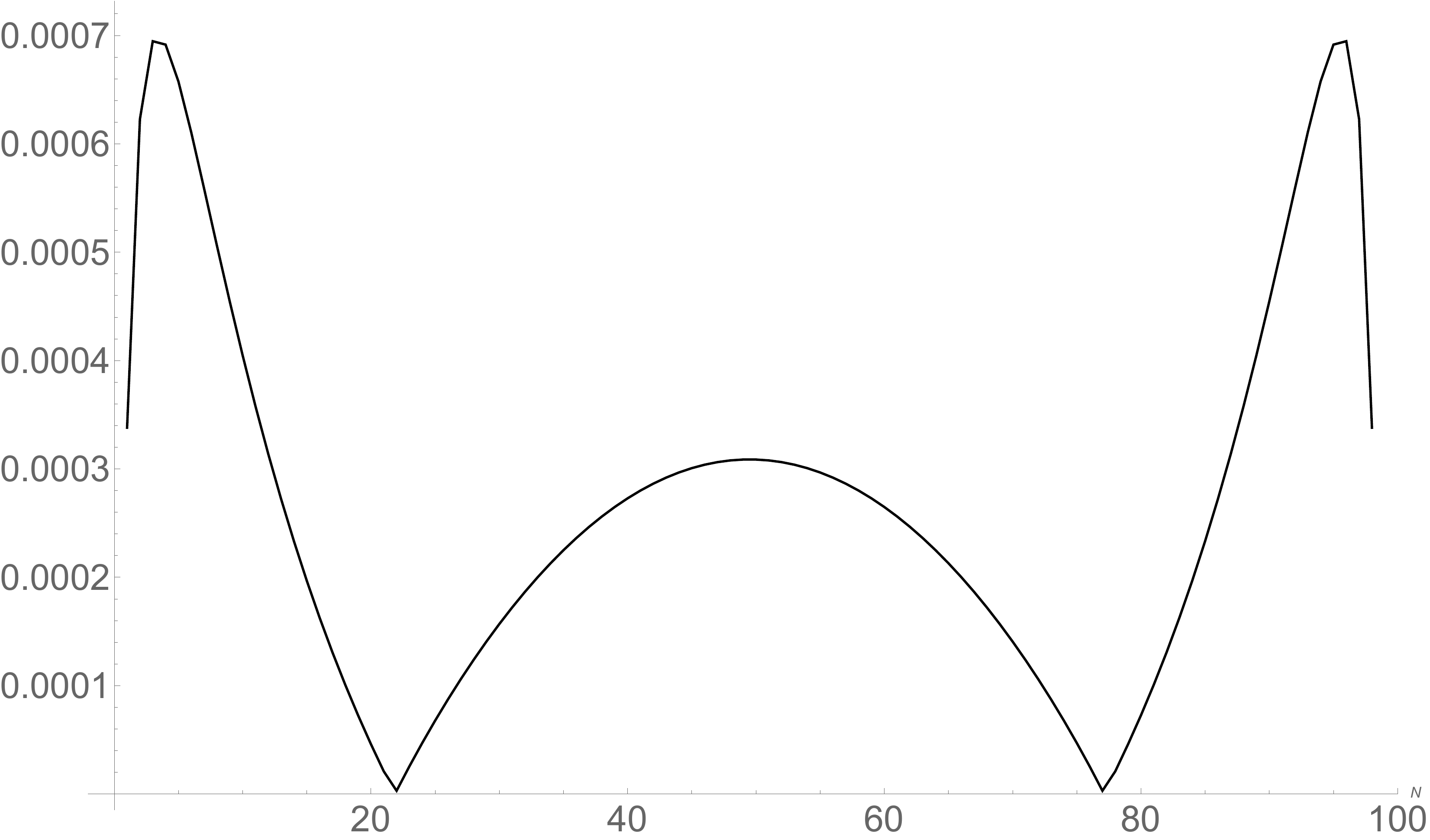} \quad
	\includegraphics[width=5cm]{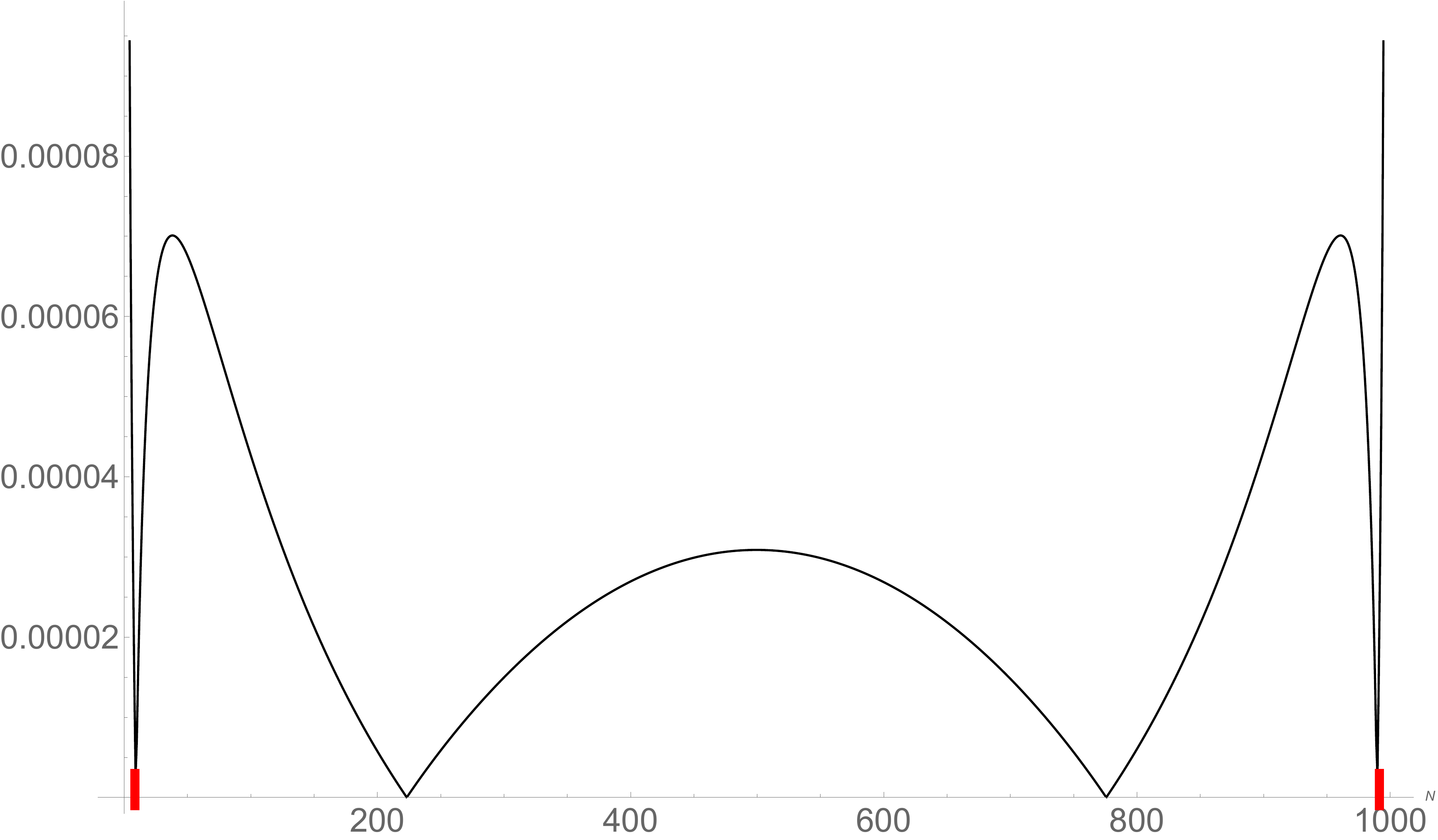} \quad
	\includegraphics[width=5cm]{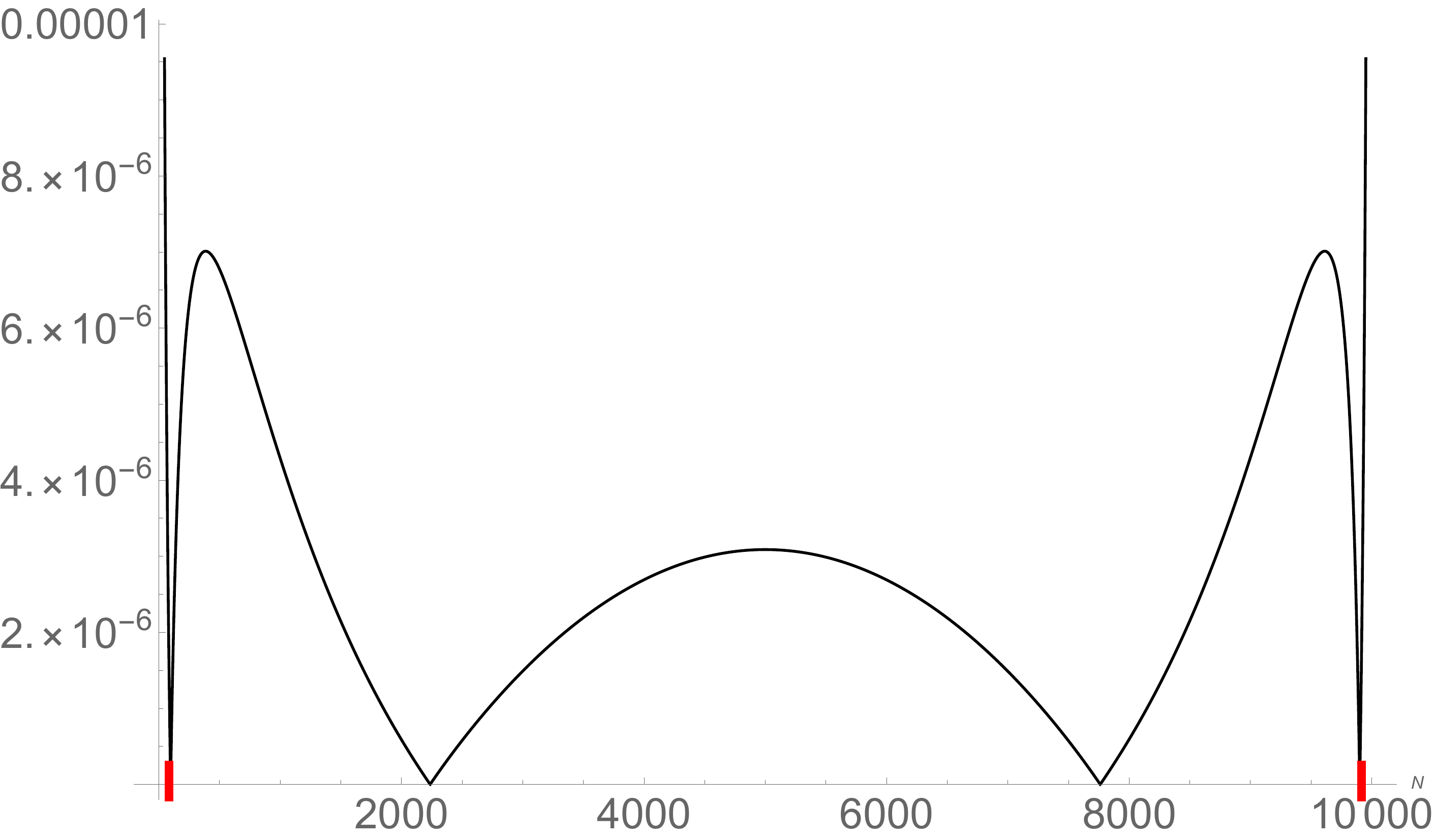}
	\caption{Deviation from $1/N$ of the volume of each partition set generated by the probabilistic values for $r$ in dimension 5. From left to right, $N=100, 1000, 10000$. The $x$-axis represents the index, $i$, of the set in the partition, while the $y$-axis shows the deviation of the volume of the $i$-th set in the approximate partition from the desired volume of $1/N$. \label{fig:volplot}}
\end{figure}

\subsection{Conclusions}
For a given $N$, if one is required to partition the $d-$dimensional unit cube into $N$ equivolume slices with partitioning lines orthogonal to the main diagonal, we suggest to proceed as follows:

\begin{itemize}
	\item If $N \leq d!$: In this case, all partitioning planes $H_r$ belong in the range $1 \leq r \leq d-1$. As our numerical experiments above confirm, in the main body of the hypercube the normal approximation \eqref{prob} generates a quasi-equivolume partition. Therefore the recommendation would be to generate the quasi-equivolume partition using Theorem \ref{normalconverge}. We note that the small errors we make in this region should play a limited role in the overall picture, particularly with respect to applications in uniform distribution theory.
	
	\item If $N > d!$: The partitioning hyperplanes now sit in every segment of the cube and we are required to consider the full range of the parameter $0 \leq r \leq d$. As in the first case, the normal approximation yields an accurate equivolume partition in the main body of the cube.
	Our numerical experiments above inform that in the extreme segments, i.e. when $0 \leq r \leq 1$ and $d-1 \leq r \leq d$ the normal approximation deviates significantly from the desired volumes of the sets. To deal with this inaccuracy, we recall the algebraic formula in Theorem \ref{algebraicsol}. In the case $k=d$ (i.e. in the initial segment $0 \leq r \leq 1$), the formula reduces to $$r^d = d! \left( 1- V^+_d(r) \right).$$ We suggest to solve this simplified equation for $r$ to generate the hyperplanes contained in the initial segment. By symmetry, we can also easily obtain the values for $r$ in the upper most segment.
\end{itemize}


\section{Looking for minimal sets}
\label{sec:num}

In this section, we use several different black-box optimisers to look for approximations of the minimal sets $\cS^{*}(N,2)$ and $\cS^{*}(N,3)$, which we then compare to our equivolume stratification. Rather than working with the hyperplane parameters $r_i \in [0,d]$ for $i = 1, \ldots, N-1$ as in the previous sections, we will now consider values $0 \leq p_1 \leq \cdots \leq p_{N-1} \leq \sqrt{d}$. These values correspond to the Euclidean distance between the origin and the intersection points of the chosen hyperplanes with the diagonal and are more convenient to use with the optimisers. Black-box optimisers generate a number of samples, evaluate the target function for these samples (in this case, the expected $\mathcal{L}_2$-discrepancy) and then update some information on the problem. This depends heavily on the type of optimiser, it can be only the best solution found so far or an underlying distribution which can be used to find the best solution. 

Black-box optimisers present two advantages. The first is that it is a very difficult task to directly compute the expected $\mathcal{L}_2-$discrepancy of a general stratification. Recent approaches have always studied square boxes aligned with the axes \cite{mf22} or used numerical experiments \cite{mf21}; note that finding the best positions for 3 points was already an involved process, see \cite{mf21}. Black-box optimisers, in contrast, do not require us to obtain new results on the \emph{structure} of the function but directly try to find good intersection points for the given setting through a systematic trial and error process.

The second advantage is that for each intersection point, we only need to keep a single parameter since we know the hyperplanes are orthogonal to the diagonal, regardless of the dimension $d$. We can model the problem as a $N-1$ dimensional optimisation problem $(p_1,\ldots,p_{N-1})$, where for each element $p_i$ we have the bounds $p_i \in [0,\sqrt{d}]$, as well as the ordering constraints on the different parameters, i.e. $p_1 \leq p_2 \leq \ldots \leq p_{N-1}$.

\begin{table}
    \includegraphics[width=10cm]{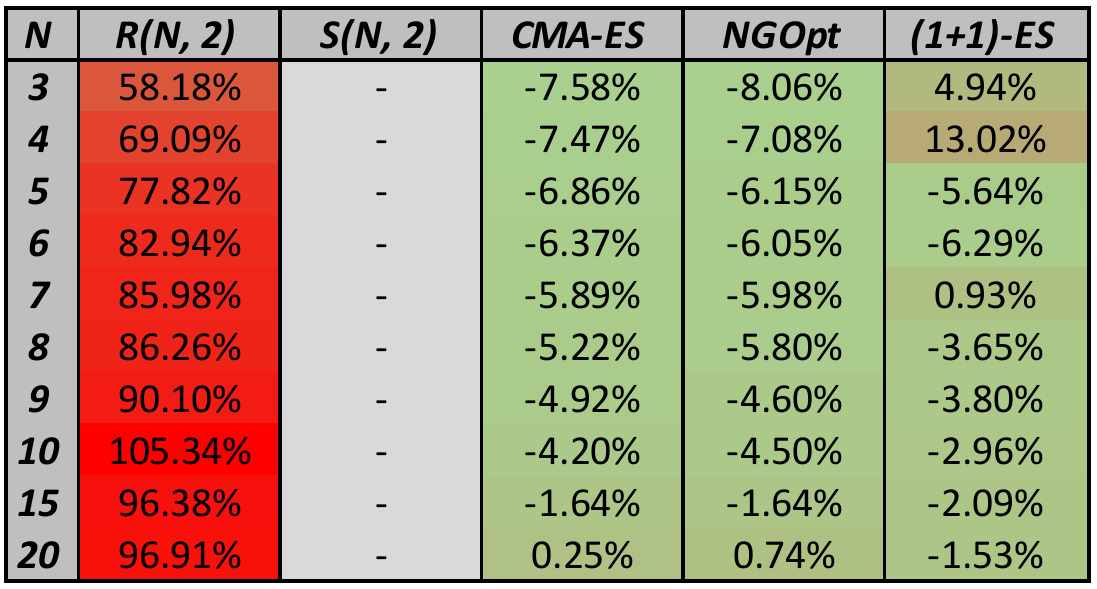}

    \vspace{6mm}
    \caption{A percentage change comparison of the expected $\mathcal{L}_2$-discrepancies of different point sets with respect to the equivolume stratification: uniformly distributed random points $R(N,2)$, the equivolume partition corresponding to $\cS(N,2)$ and the best sets obtained by the three black-box optimisers. Expected $\mathcal{L}_2$-discrepancy values are done with 10000 repetitions to guarantee a better precision. We refer to Table \ref{Compare2} in Appendix \ref{app3} for the exact expected discrepancy values.}
    \label{fig:percentagecomparison}
\end{table}

\subsection{Experiment setup} 
Regardless of the chosen optimiser -- see Appendix \ref{app2} for a description of the different algorithms: CMA-ES, NGOpt and (1+1)-ES -- the general model is identical. In dimension $d$, we have an $N-1$ dimensional problem, where each parameter is in the interval $[ 0,\sqrt{d} ]$. The optimiser is given a budget of 1000 evaluations where each evaluation consists in approximating the expected $\mathcal{L}_2$-discrepancy of the corresponding stratification. To approximate this function, we first sort the different parameters to obtain a valid stratification. We sample randomly a point in each stratum then compute the $\mathcal{L}_2$-discrepancy of the resulting point set. This is repeated 1500 times and averaged to obtain an approximation of the expected $\mathcal{L}_2$-discrepancy.

All experiments were run in dimension 2 and 3 in a low-fidelity setting, with a high-fidelity correction at the end. Indeed, since the optimisers are tracking only the best value found so far, it is possible that the returned point sets have slightly higher discrepancy and the calculation during the optimiser run was one with a high variance. To correct for this phenomenon, we recompute the expected $\mathcal{L}_2$-discrepancies of the corresponding point sets, this time with 10000 repetitions to guarantee a better precision. No initialisation is provided to the optimisers and 10 runs are made for each instance.

Note that we could extend our experiments beyond dimension 3. The main difficulty is sampling uniformly from given hypercube slices. In our experiments, we define a slice based on the distance between two successive points $p_i$ and $p_{i+1}$, rotate and translate this slice into the right position and then sample a point in it. Note that re-sampling may be necessary as usually not all of the slice is indeed contained in the unit cube. However, this is in general not a problem. For $d=3$ we use quaternions to define the rotations. In higher dimensions this gets more involved.

Both the {\tt ioh} \cite{IOHexperimenter} and {\tt Nevergrad} \cite{Nevergrad} Python packages were used for the optimisers, the implementations are taken from the { \tt Nevergrad} package. Experiments were run in Python, using the {\tt random} module to generate the points. The kernel density estimations were obtained using the {\tt sklearn} package and in particular {\tt Kernel Density}.

\subsection{Experimental results}

Table \ref{fig:percentagecomparison} gives a percentage comparison of the expected $\mathcal{L}_2$-discrepancy of the equivolume partition $\cS(N,2)$, of uniformly distributed random points $R(N,2)$, and of the best distributions found by the different optimisers, using $\cS(N,2)$ as a baseline. The discrepancy estimation is obtained with 10000 repetitions.

As already empirically observed in \cite{mf21}, the expected $\mathcal{L}_2$-discrepancy of the equivolume partition is approximately smaller by a factor of 2 than the discrepancy of uniformly distributed random points for all different set sizes tested. All three black-box optimisers return very similar values, which are about 7\% better for the lowest values of $N$ after re-evaluation of the expected $\mathcal{L}_2$-discrepancy. Since all three optimisers return similar values, this suggests that the point ordering in the problem structure was not an issue, at least for these low values of $N$. The only outliers are for very low values of $N$, where the (1+1)-ES is struggling: this seems to be largely due to the low-fidelity optimisation, the best discrepancy values found were around 0.025 for $N=3$ before correction. We also note that for $N=4$, both the discrepancy values and the point sets returned are quite similar to the improved construction given in \cite{mf21} (see Table \ref{points2d} for the point sets). The equivolume partition performs better than the optimisers for $N=20$ and its relative performance improves when $N$ increases in general.

As mentioned previously, discrepancy values for the point sets returned by the optimisers were recomputed with a higher precision. This implies that possibly better point sets were \emph{overseen} during the optimisation because of the imprecision in the expected $\mathcal{L}_2$-discrepancy calculations. The gradually decreasing gap, both for the returned discrepancy value and the corrected one, also suggests that our optimisers may be struggling with the higher dimensional problem.

Table \ref{points2d} gives examples of the obtained point sets (see also Appendix \ref{app3}).
\begin{table}[]
    \centering
    \begin{tabular}{|c|c|}
    \hline
        N & Point set obtained by CMA-ES  \\
        \hline
        3 & [0.525326, 1.094506] \\
        \hline
        4 & [0.416424, 0.880939, 0.995149]\\
        \hline
        5 & [0.394111, 0.617818, 0.967749, 1.021636]\\
        \hline
        6 & [0.346292, 0.547769, 0.833453, 0.916769, 1.066748]\\
        \hline
        7 & [0.325707, 0.506785, 0.6853, 0.876577, 0.97997, 1.087066]\\
        \hline
        8 & [0.322496, 0.482121, 0.583169, 0.863154, 0.874371, 0.981856, 1.066117]\\
        \hline
        9 & [0.293892, 0.448608, 0.525631, 0.728415, 0.864673, 0.902607, 0.992526, 1.097269]\\
        \hline
        10 & [0.295754, 0.447257, 0.49142, 0.603255, 0.805188, 0.88467, 0.923663, 1.005221, 1.120071] \\
         \hline
        15 & [0.279248, 0.27951, 0.434256, 0.541966, 0.560045, 0.623614, 0.7019, \\ 
        & 0.79564, 0.798366,  0.814506, 0.931907, 0.958269, 1.010601, 1.333425] \\
        \hline
 		20 & [0.14877, 0.239469, 0.432346, 0.471298, 0.487274, 0.50509, 0.545568, 0.58021, 0.612064, \\
		& 0.737669, 0.802694, 0.856641, 0.889084, 0.921972, 0.946705, 0.977818, 1.009777, 1.022098, 1.259538] \\
         \hline
    \end{tabular}
	\vspace{5mm}
    \caption{Optimal partitioning points returned by CMA-ES. Note that we need $N-1$ hyperplanes for a set of $N$ points.}
    \label{points2d}
\end{table}

Interestingly, all the sets returned by the three optimisers have a similar structure once $N$ is big enough ($N\geq 5$). 
This structure can be visualised using a \emph{kernel density estimate} with a Gaussian kernel as illustrated in Figure \ref{kernels2}.
\begin{figure}[h!]

  \includegraphics[width=0.32\textwidth]{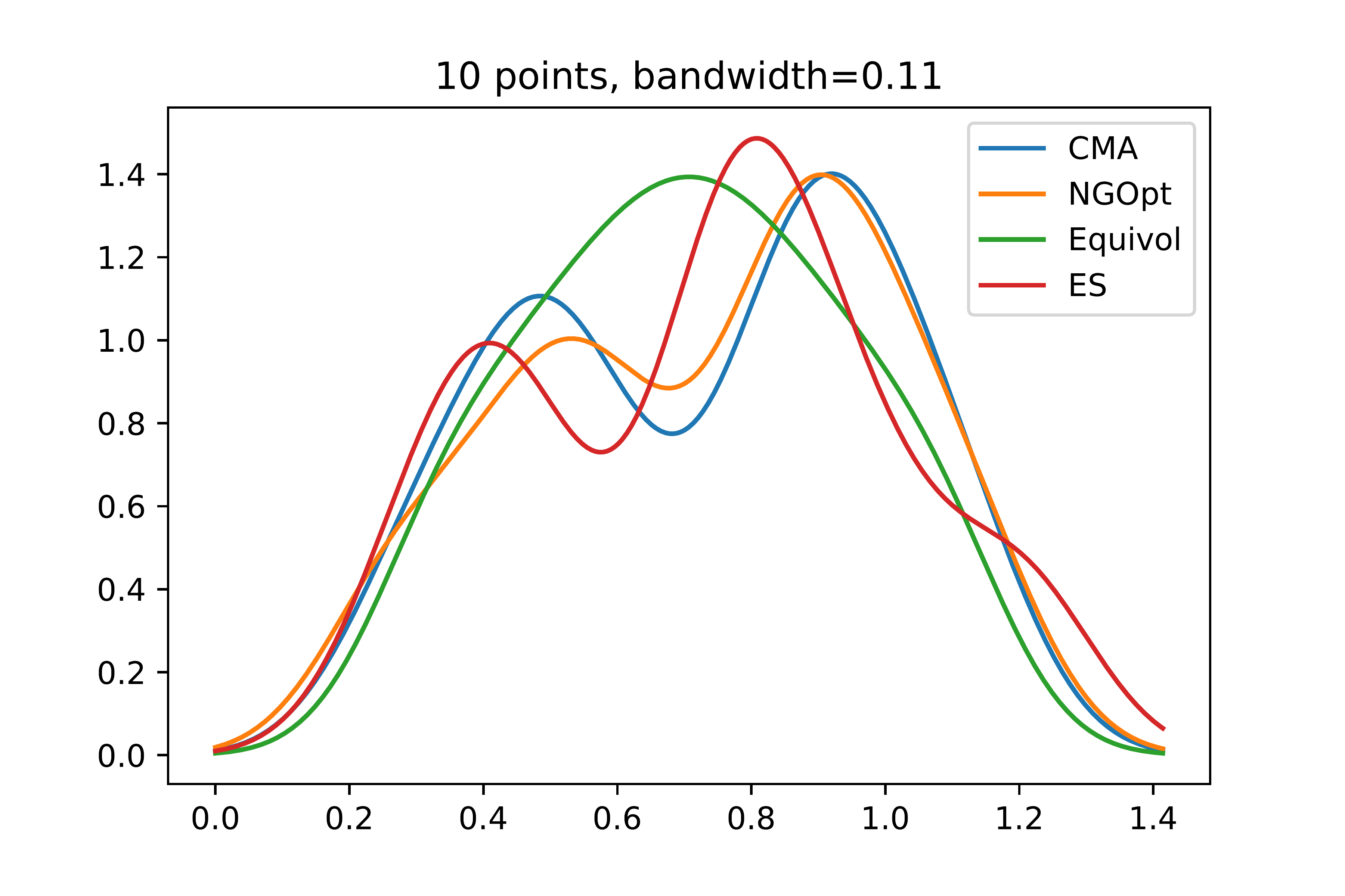} \hfill
  \includegraphics[width=0.32\textwidth]{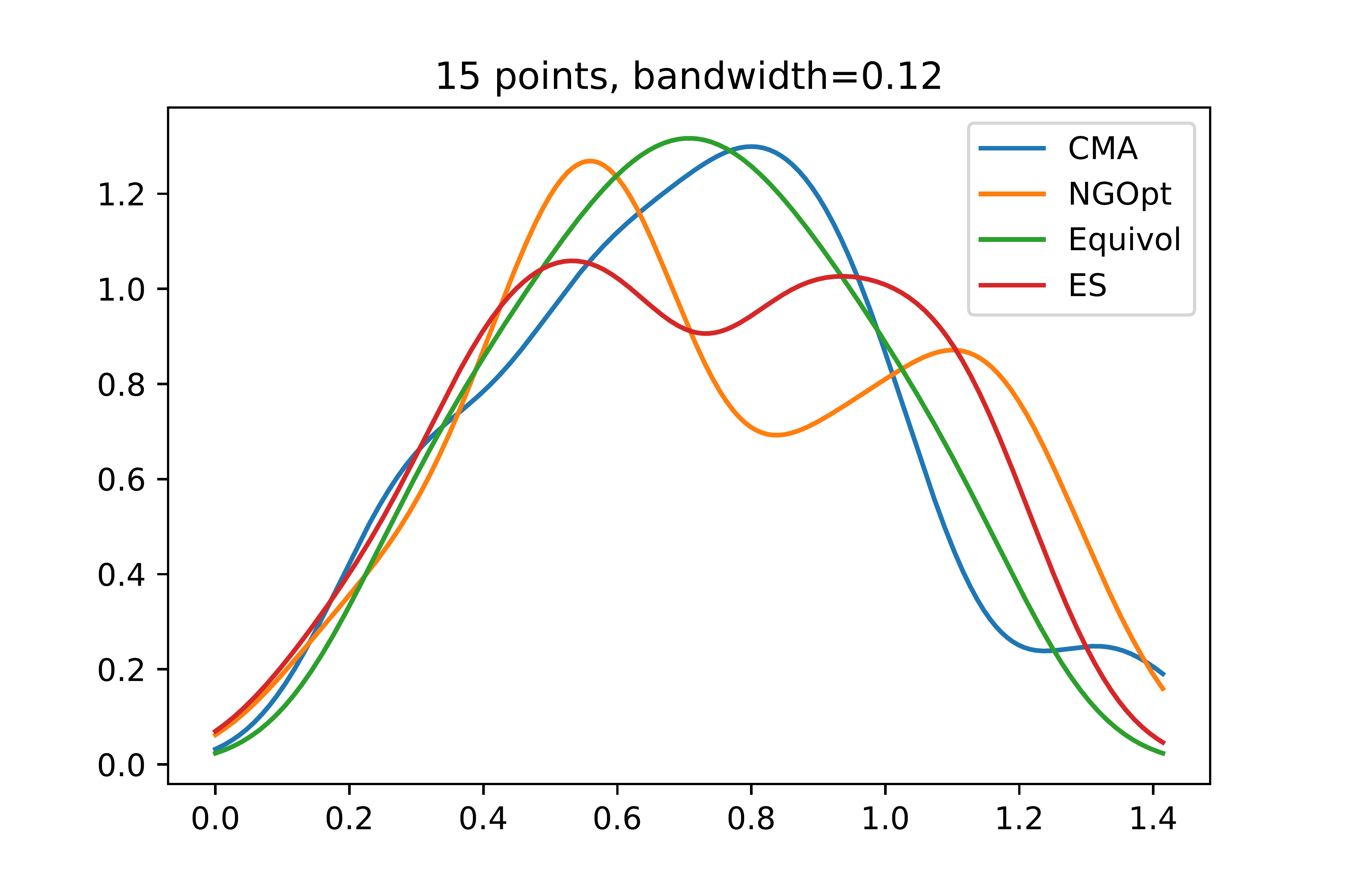} \hfill
  \includegraphics[width=0.32\textwidth]{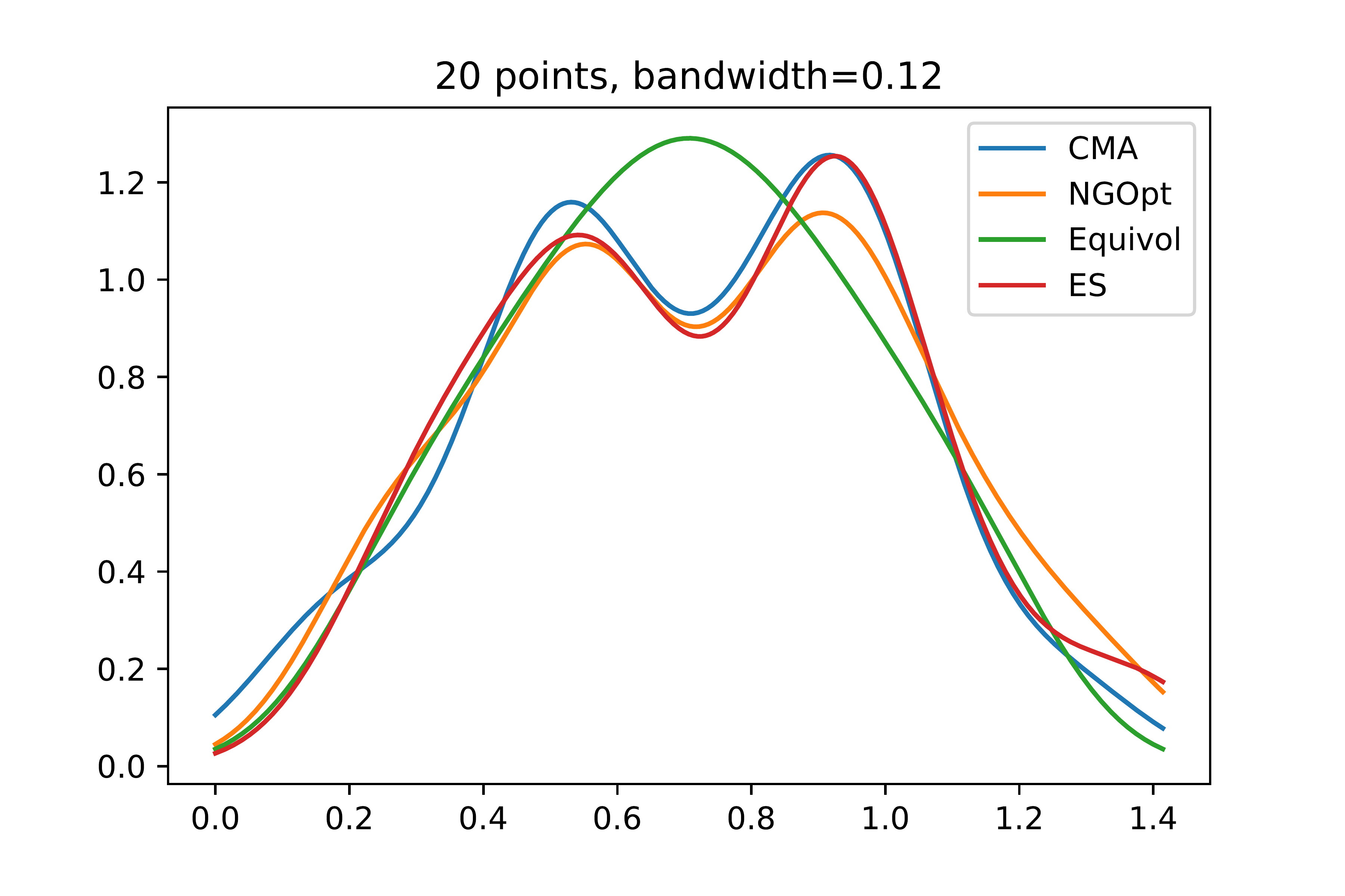} \hfill
  \caption{Kernel density estimate plots with Gaussian kernels for the best partitioning points obtained by the optimisers in dimension 2 as well as the equivolume sets for $N=10,15,20$.}
  \label{kernels2}
\end{figure}

The point sets seem to have two clusters of points on the diagonal. 
The spread of the points seems bigger meaning that the first point has a smaller value as the first of the equivolume set, while the last point has a larger value as the corresponding point of the equivolume set. 

Turning to dimension 3, Table \ref{comparison3d} gives the percentage comparison of the expected $\mathcal{L}_2$-discrepancy of the different methods with $\cS(N,3)$ as the baseline. The results are quite similar to the $d=2$ case, where all three black-box optimisers returning similar values, all outperforming the equivolume stratification and with the (1+1)-ES giving the worst results. Once again, we notice that the relative performance of the equivolume strategy improves as $N$ increases. The plots in Figure \ref{Plots3d} are similar to corresponding plots in dimension 2: there are clusters of strata and the spread of the points seems larger. However, the last diagonal point $p_{N-1}$ in the optimisers' sets is not always larger than the last diagonal point for $\cS^*(N,3)$. We visualise the point sets again with a kernel density estimate in Figure \ref{kernels3}.

\begin{table}[h!]

\includegraphics[width=10cm]{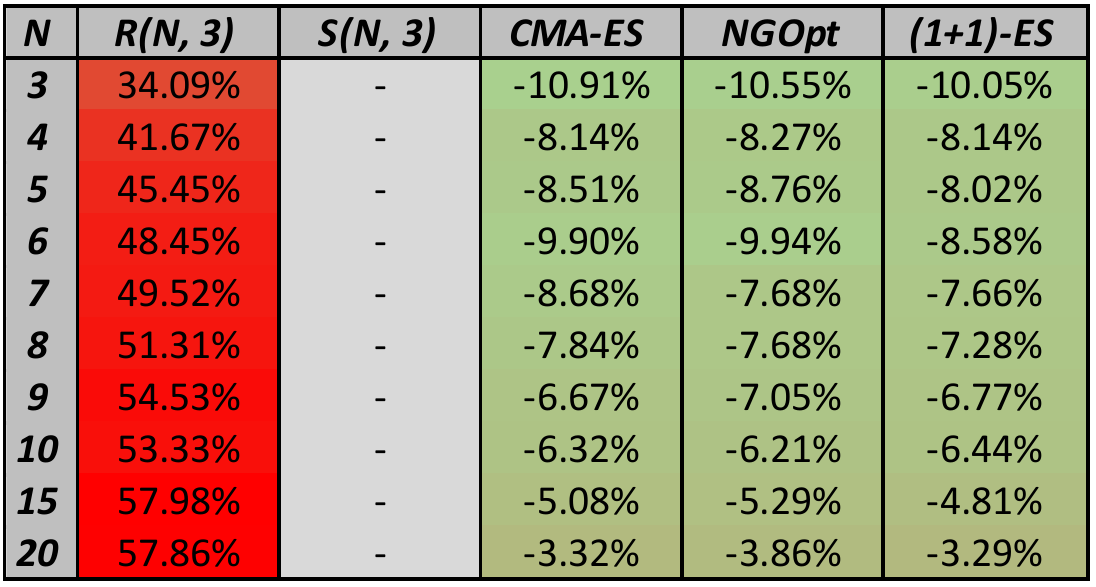}

\vspace{6mm}
\caption{ A percentage change comparison of the expected $\mathcal{L}_2-$discrepancy values of different point sets using $\cS(N,3)$ as a baseline. Discrepancy values are done with 10000 repetitions; we refer to Table \ref{Compare3} in Appendix \ref{app3} for the exact expected discrepancy values. \label{comparison3d}}
\end{table}

\begin{figure}[h!]

  \includegraphics[width=0.3\textwidth]{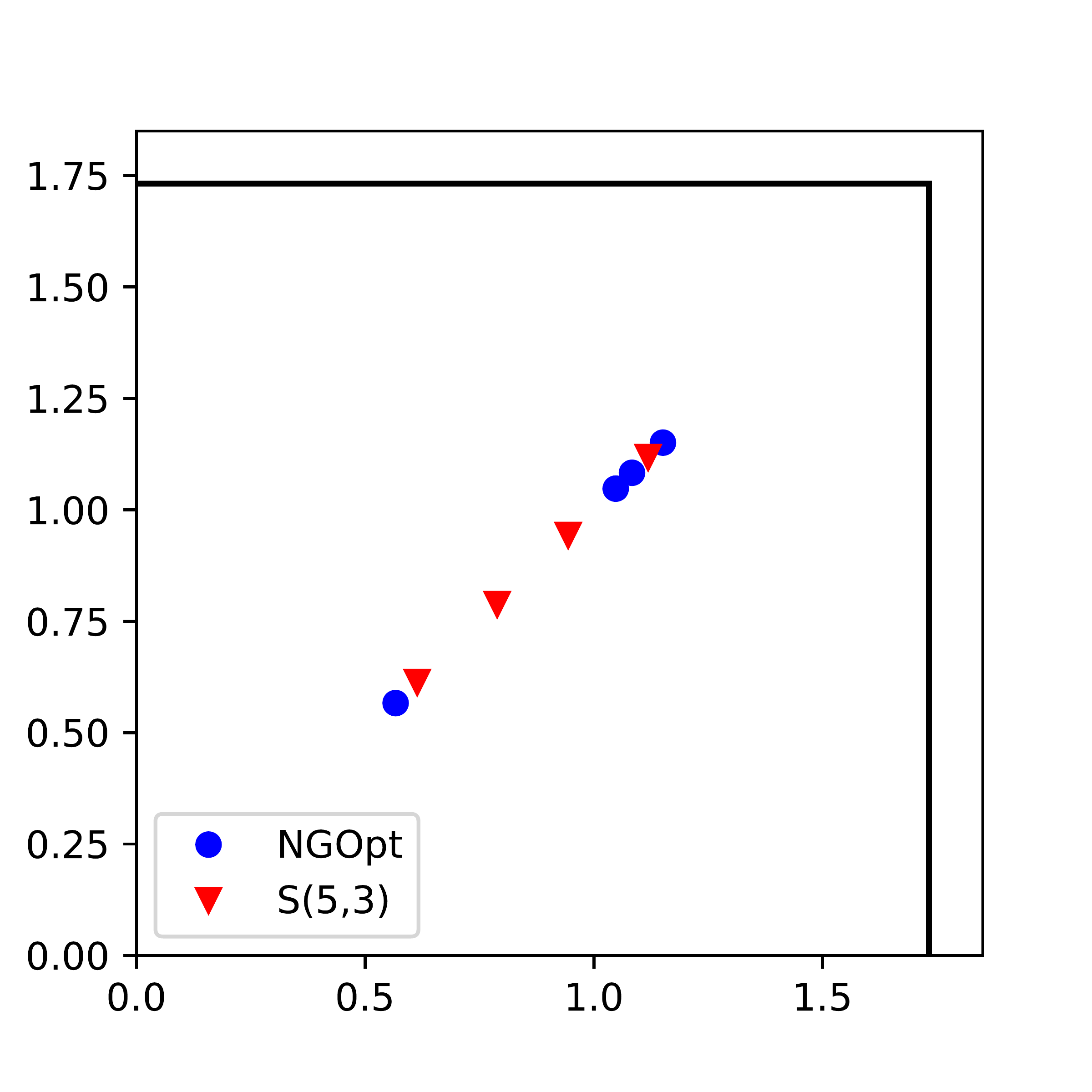} \hfill
  \includegraphics[width=0.3\textwidth]{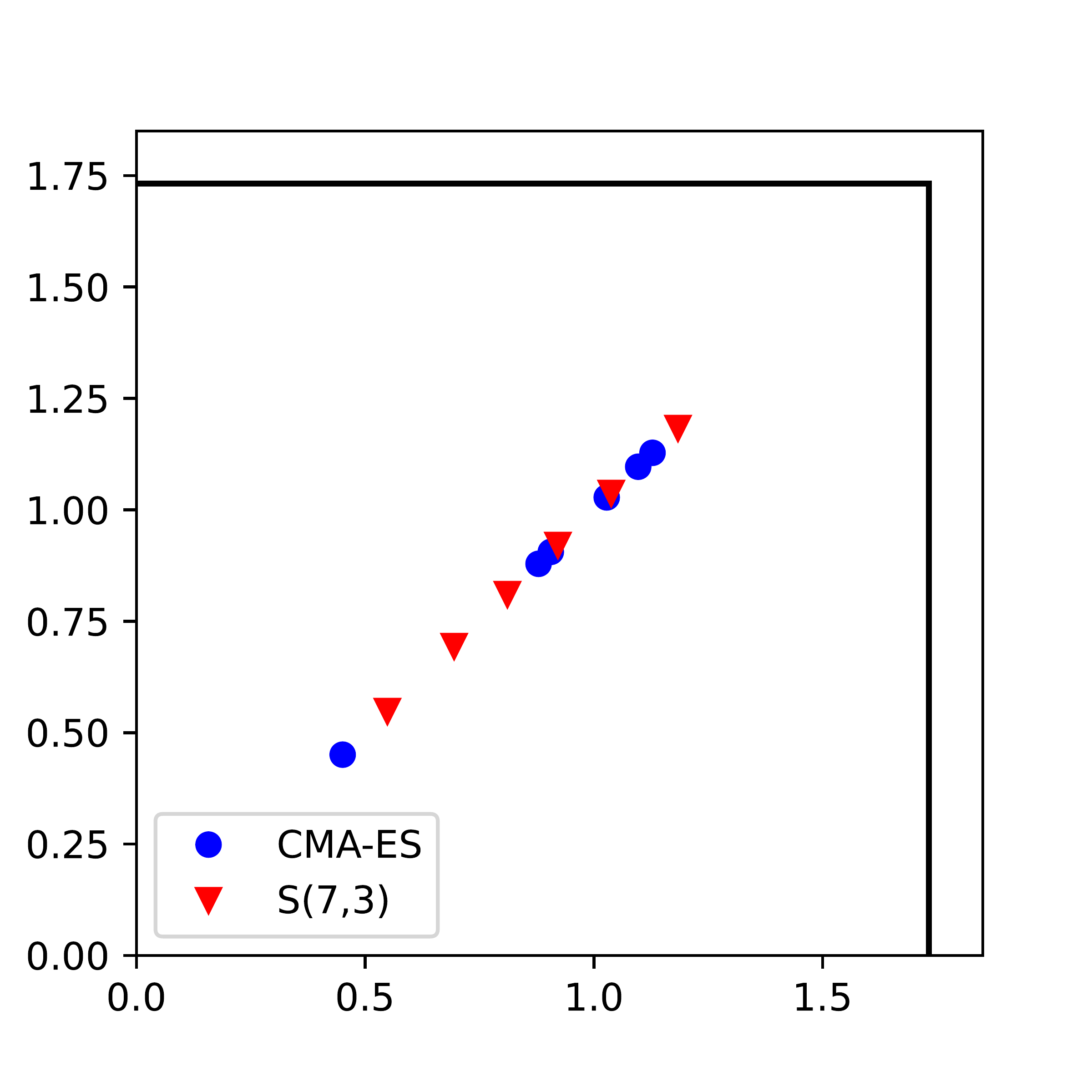} \hfill
  \includegraphics[width=0.3\textwidth]{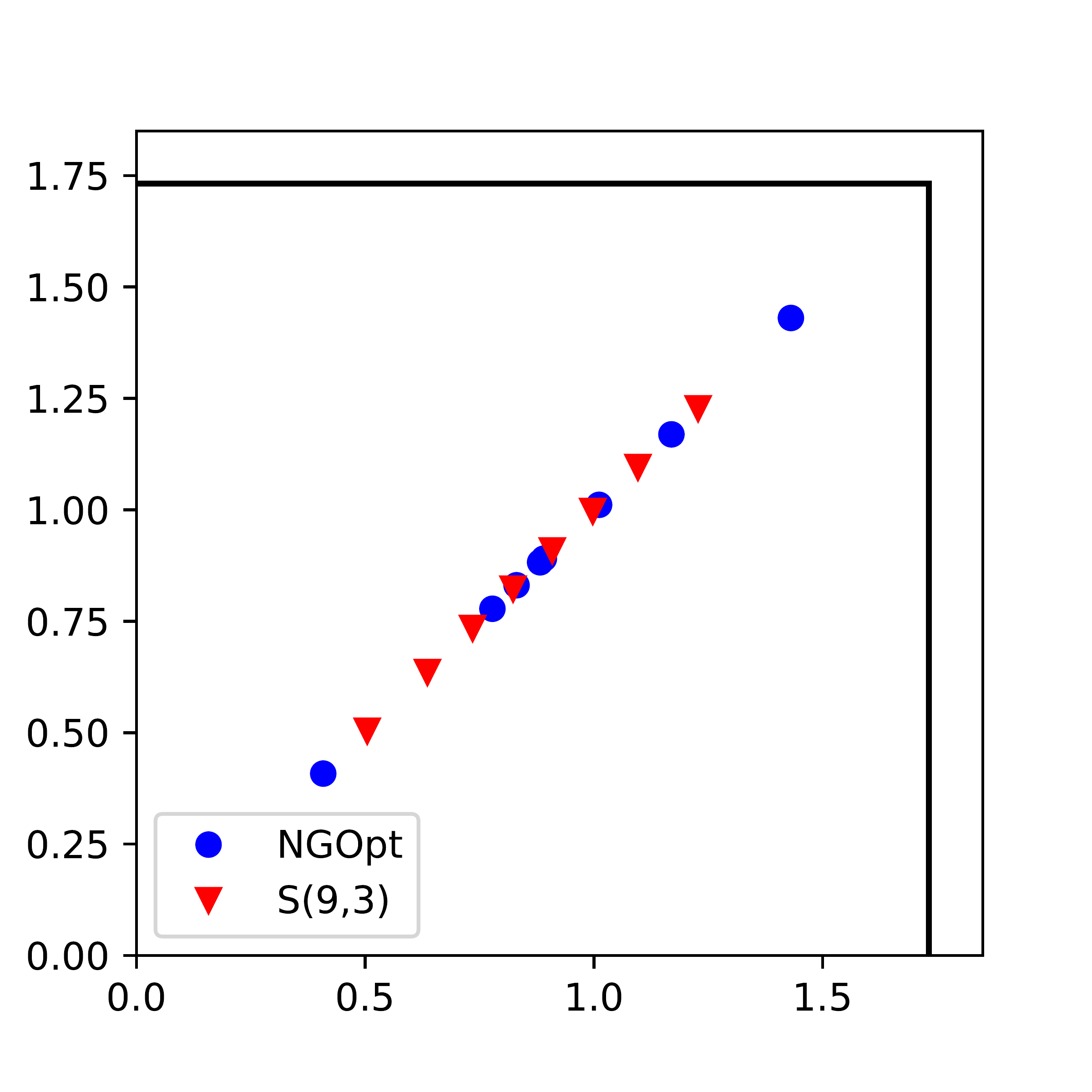} \hfill
  \medskip
  \includegraphics[width=0.3\textwidth]{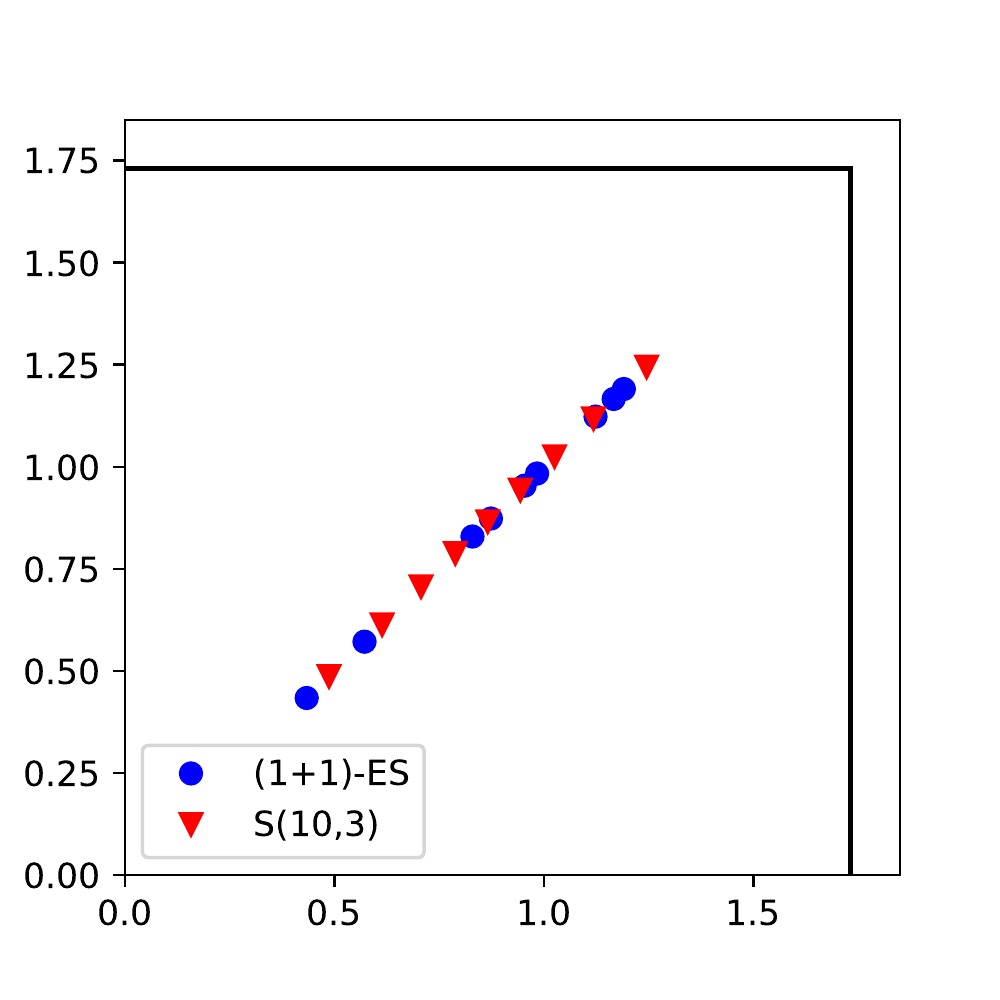}\hfill
  \includegraphics[width=0.3\textwidth]{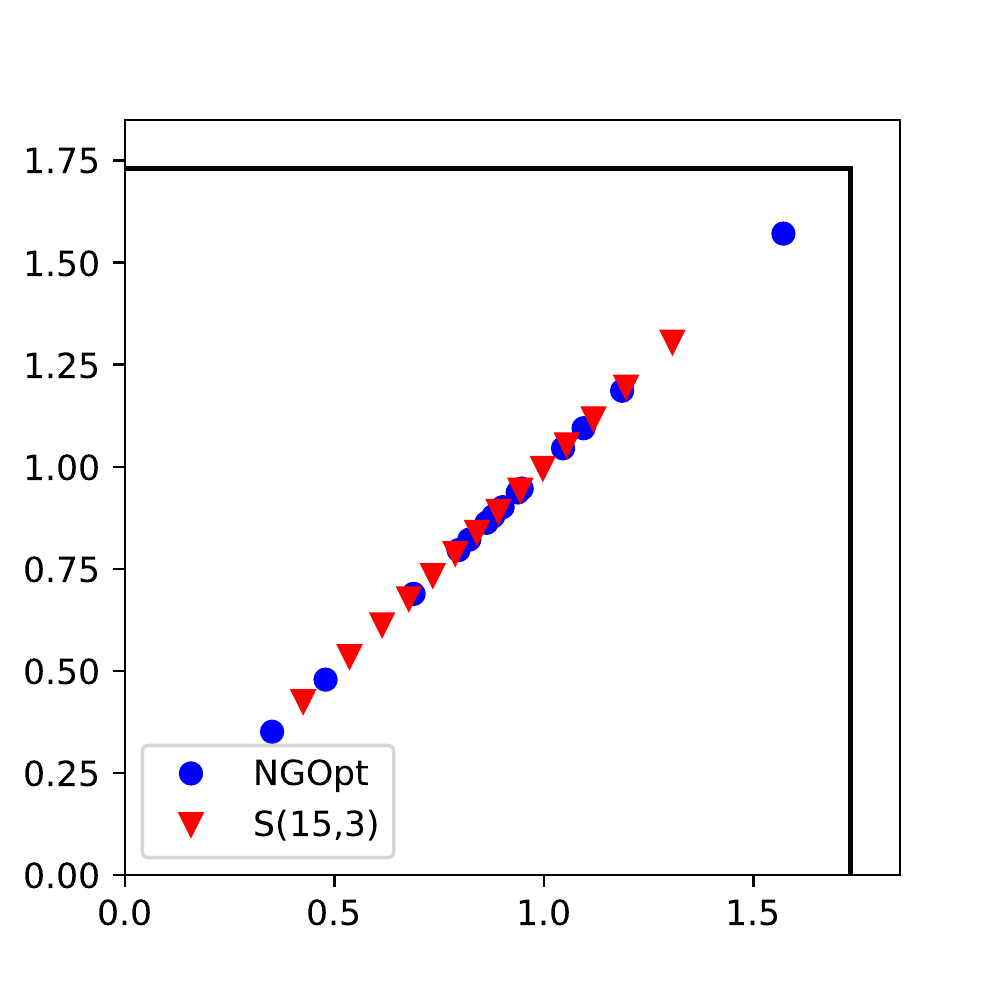} \hfill
  \includegraphics[width=0.3\textwidth]{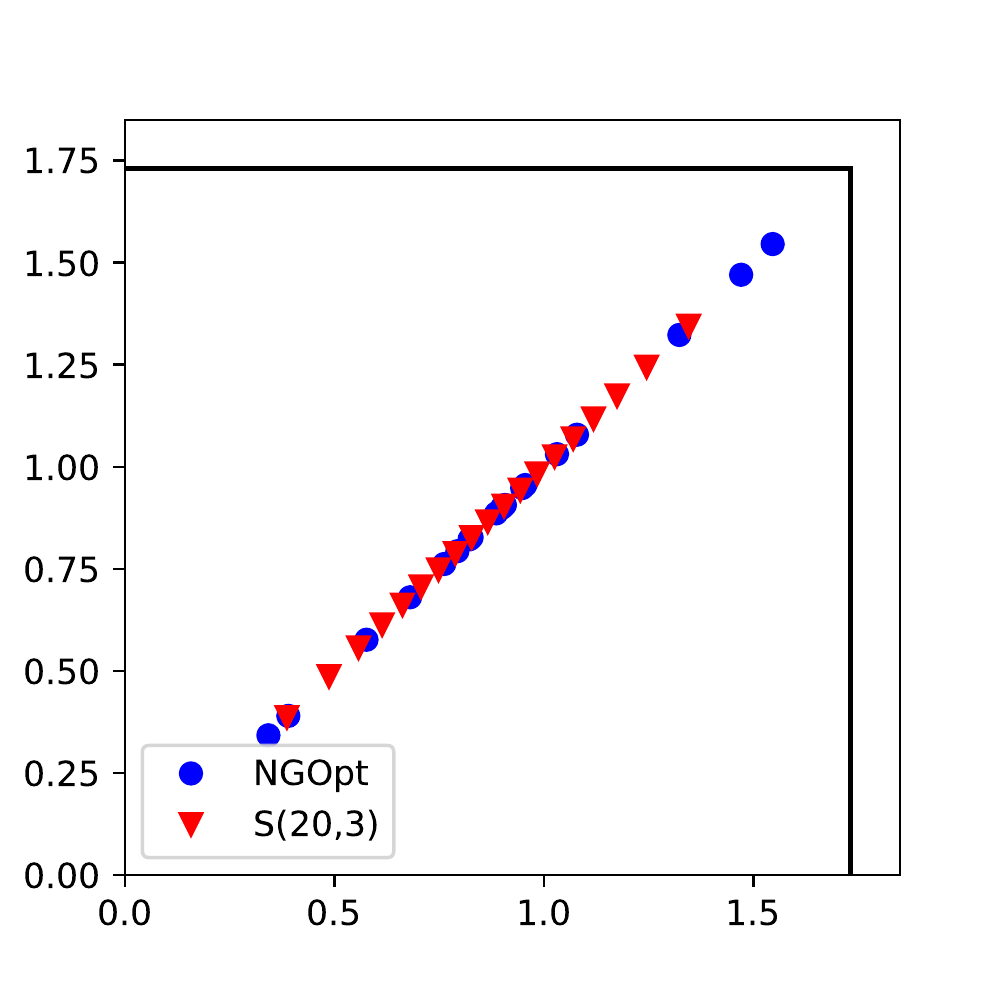} \hfill
  \caption{Best partitions generating stratified point sets obtained by the optimisers in dimension 3 compared to the equivolume partition for $N=5,7,9,10,15,20$.}
  \label{Plots3d}
\end{figure}

\begin{figure}[h!]

  \includegraphics[width=0.32\textwidth]{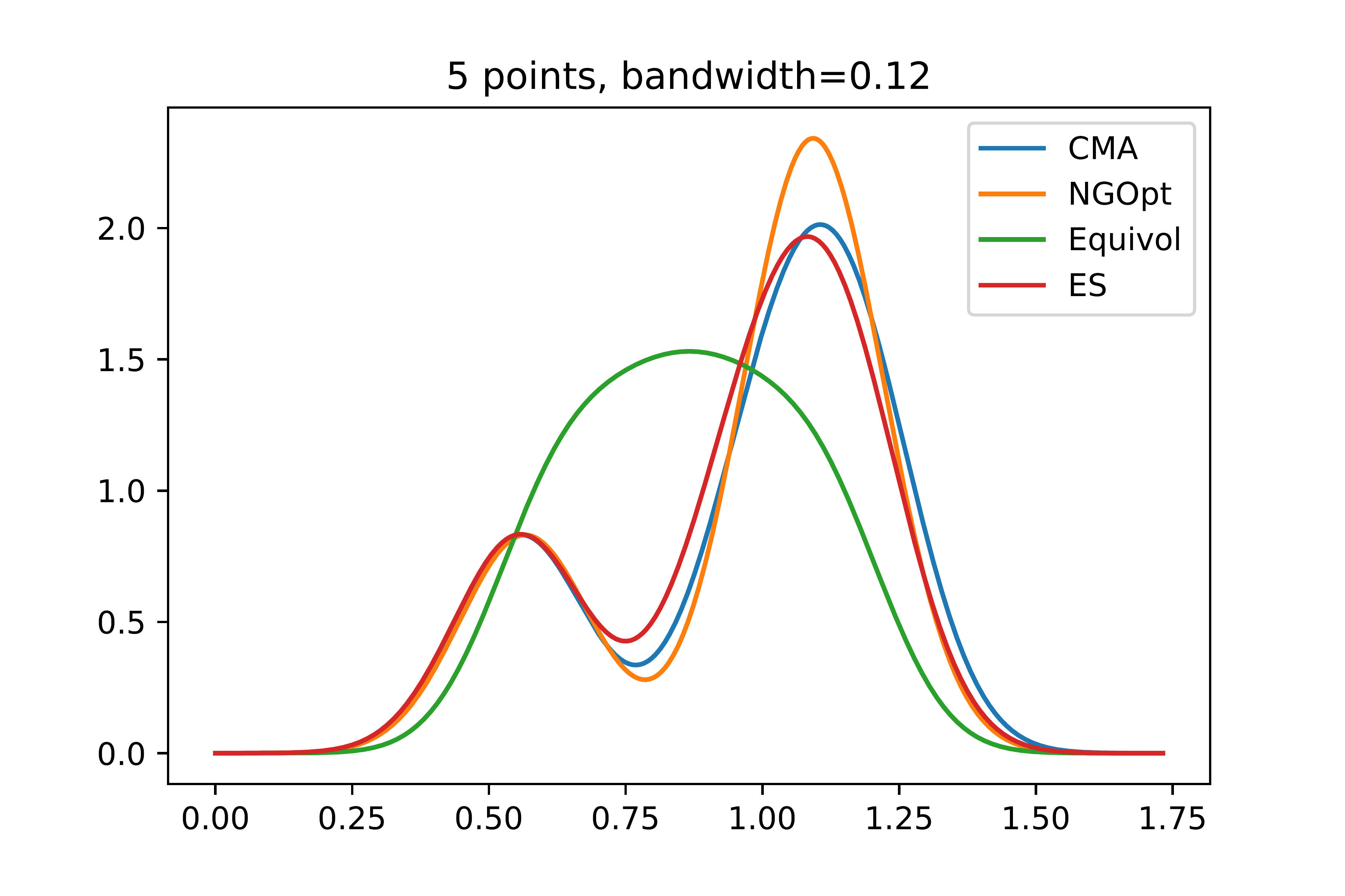} \hfill
  \includegraphics[width=0.32\textwidth]{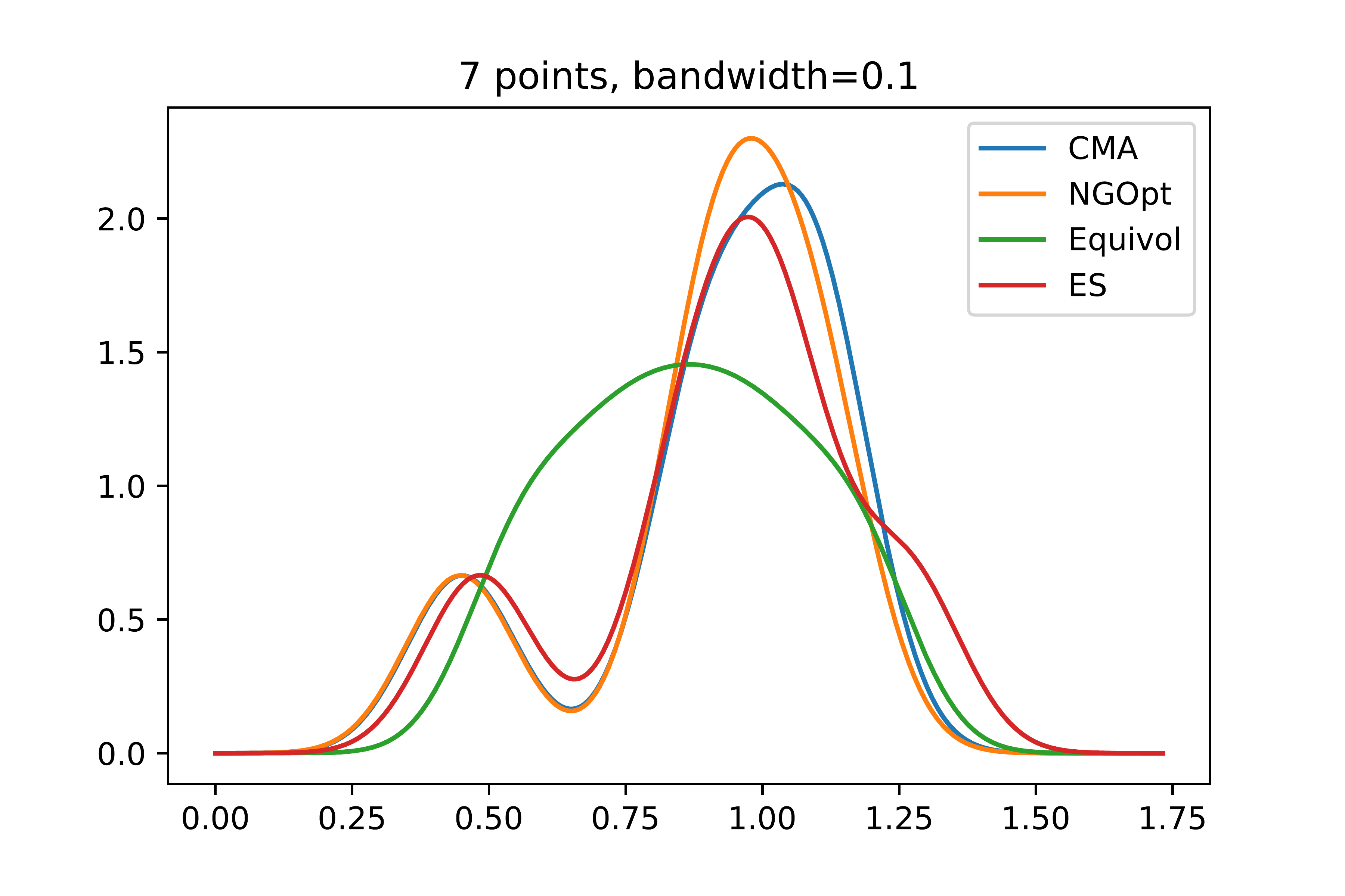} \hfill
  \includegraphics[width=0.32\textwidth]{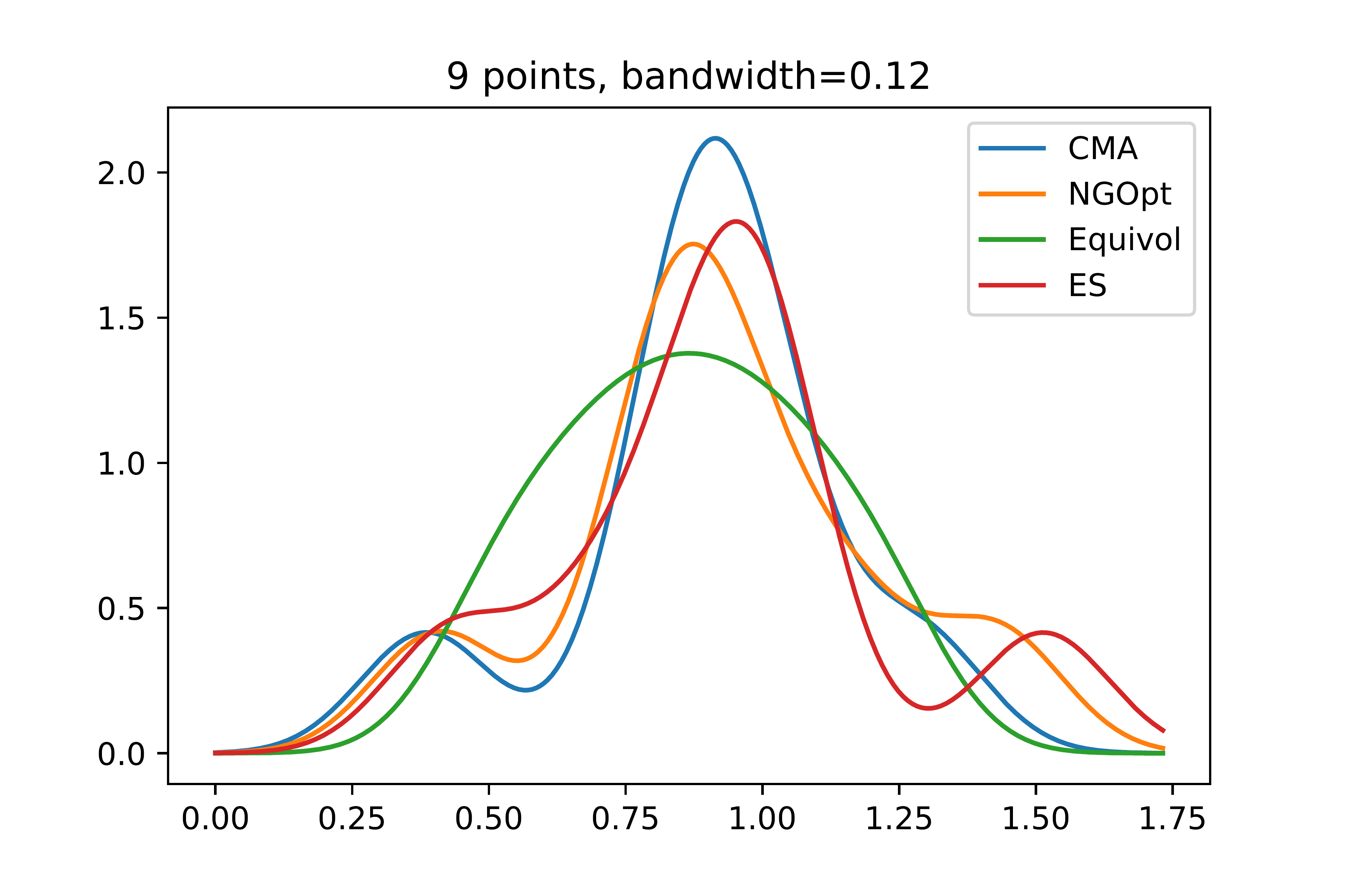} \hfill
 \medskip
 \includegraphics[width=0.32\textwidth]{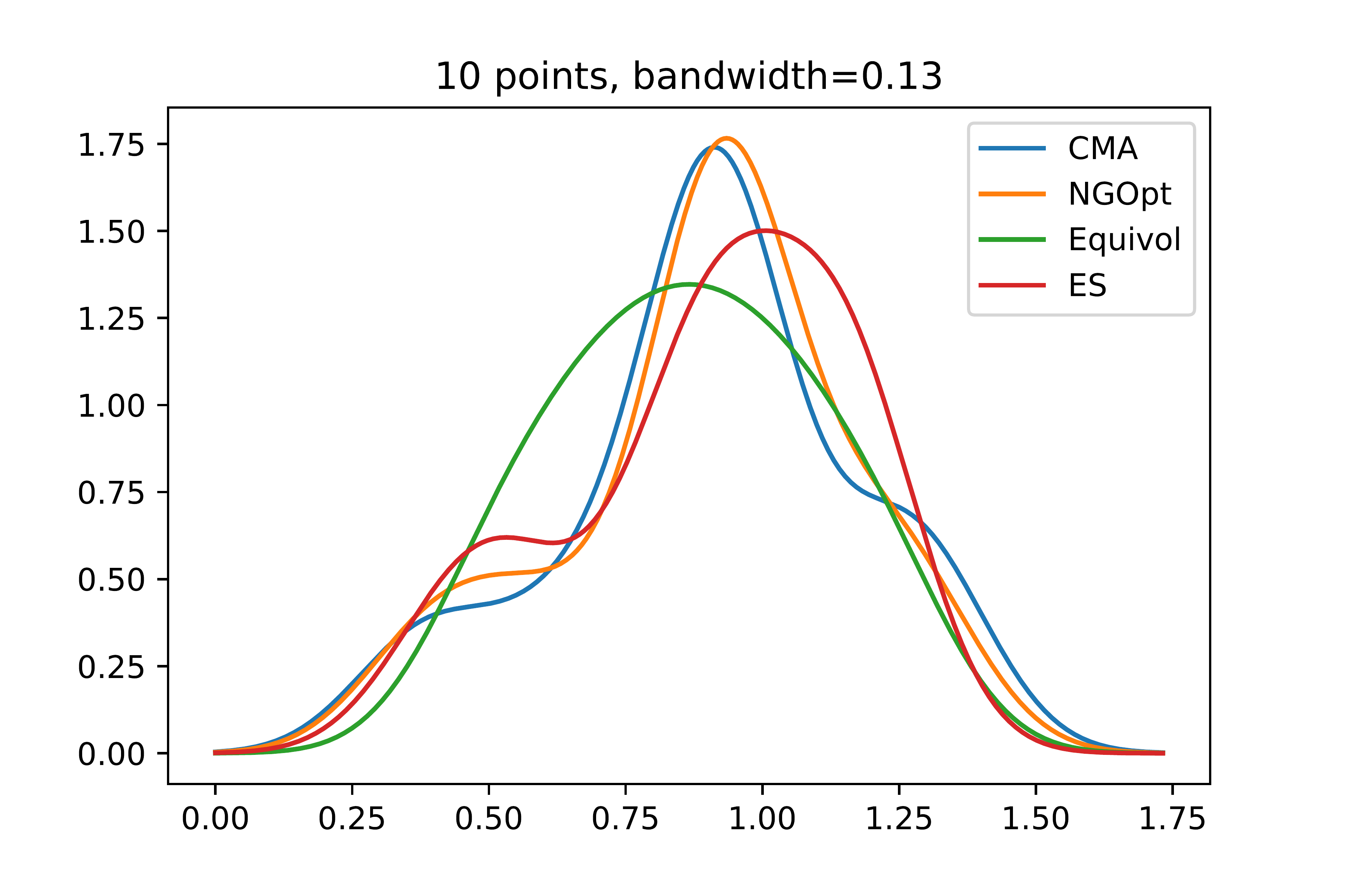} \hfill
  \includegraphics[width=0.32\textwidth]{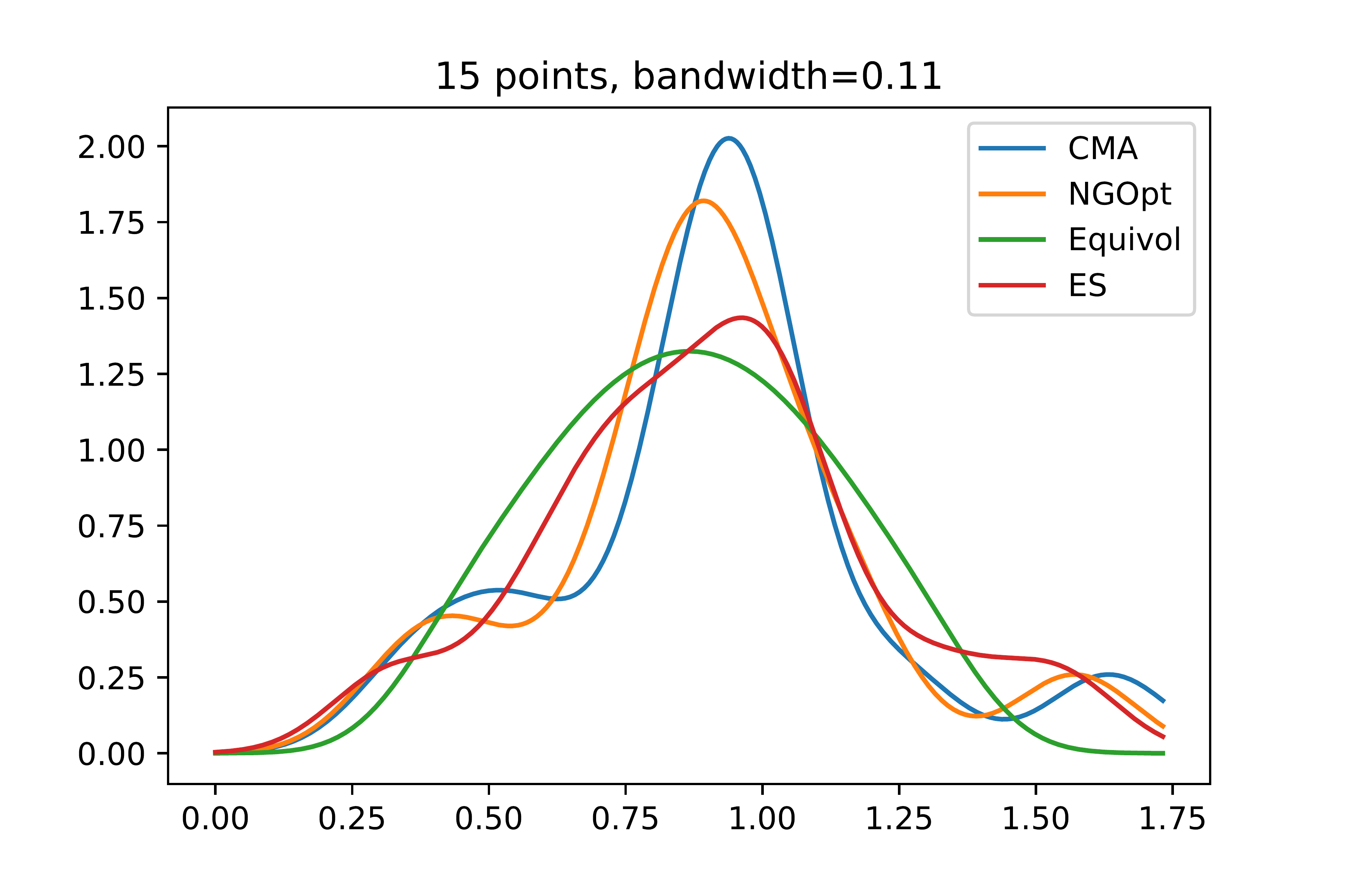} \hfill
  \includegraphics[width=0.32\textwidth]{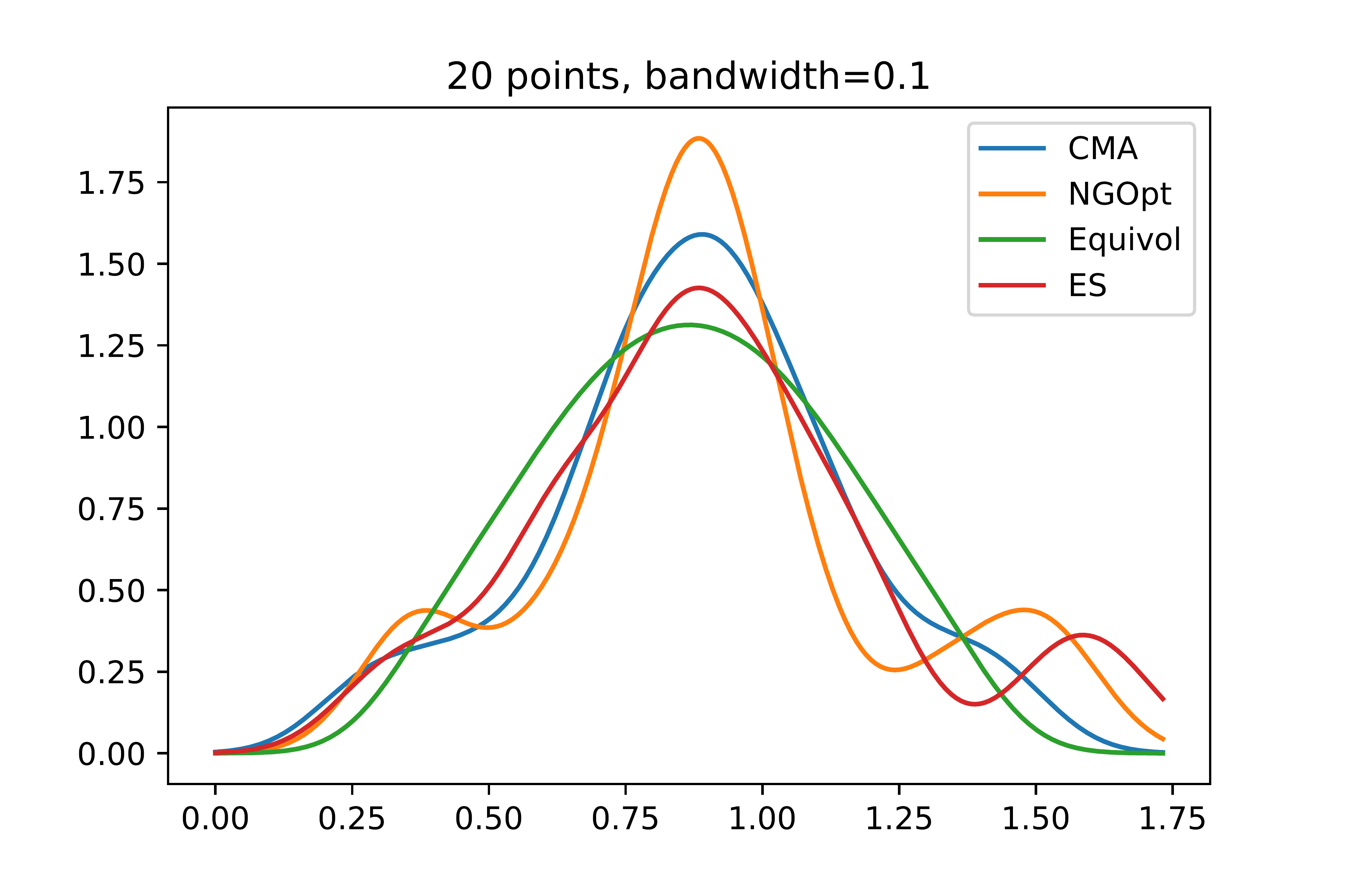} \hfill
  \caption{Kernel density estimate plots with Gaussian kernels for the best partitioning points obtained by the optimisers in dimension 3 as well as the equivolume sets for $N=5,7,9,10,15,20$.}
  \label{kernels3}
\end{figure}

Overall, while the equivolume stratification provides much better expected $\mathcal{L}_2$-discrepancy than random uniformly distributed points, our results suggest that it can still be improved. 
Finally, it is important to note that our black-box optimisers seem to return sets with similar structure as illustrated in Figure \ref{kernels2} and Figure \ref{kernels3} while the individual points in the respective sets are quite different; see tables in Appendix \ref{app3}. This suggests a plateau-like structure of the underlying space which is a blessing and a curse at the same time. One the one hand, we can safely assume that we found almost optimal solutions, on the other hand it leaves little hope to actually determine the minimal set.

\pagebreak
\vspace{3mm}
\paragraph{\bf Open questions} This paper presents results on optimal stratification of $[0,1]^2$ for the class of partitions whose partitioning hyperplanes are orthogonal to the main diagonal of the unit cube. Hence, 
\begin{enumerate}
    \item how much can one improve upon the expected discrepancy values  of the stratified point sets obtained in this text by relaxing some, or all of the constraints on the partitioning lines?
    \item what is the optimal partition (with respect to the expected discrepancy of the resulting stratified point set) of $[0,1]^d$ for $d \geq 2$?
\end{enumerate}

\section*{Acknowledgements}

The first author would like to thank Carola Doerr, Koen van der Blom and Diederick Vermetten for their help with choosing good black-box optimisers.



\appendix
\section*{Appendices}
\renewcommand{\thesubsection}{\Alph{subsection}}

\subsection{Theorem of Berry-Esseen}\label{app1}
In this short appendix we recall the currently smallest constant in the Berry-Esseen Theorem making the rate of convergence in Theorem \ref{normalconverge}  explicit.

\begin{theorem}[Berry-Esseen Theorem]
	Suppose that $X_1, X_2, \ldots, X_n$ are i.i.d. and let $\mu=\mathbb{E}X_1$, $\sigma^2 = \Var(X_1) = \mathbb{E}\left[ (X_1-\mu)^2\right]$ and $\rho=\mathbb{E}\left[ |X_1 - \mu|^3 \right]$ is the $3^{rd}$ absolute central moment. Let $R= \frac{1}{n} \sum_{i=1}^n X_i$ and denote by $$F_{Z_n}(t) = \mathbb{P}\left( \sqrt{n} \frac{R-\mu}{\sigma} \leq t \right)$$ the cumulative distribution function (CDF) of the normalised sample average. If $\rho < \infty$, then 
	
	\begin{equation*}
		\sup_{x} \left| F_{Z_n}(x) - \Phi(x) \right| \leq \frac{C \rho}{\sigma^3 \sqrt{n}}
	\end{equation*}
	
	where $\Phi$ is the CDF of the standard normal distribution and $C$ is an absolute constant.
\end{theorem}

The currently best bound for the constant is $C < 0.4748$; see \cite{Shev}. Hence, noting that for our i.i.d random variables $X_1, \ldots, X_d$ we have $\mu = 1/2$, $\sigma = 1/(2\sqrt{3})$ and $\rho = 1/32$  we can implement the Berry-Esseen Theorem to see that the rate of convergence noted in Theorem \ref{normalconverge} is

\begin{equation}\label{rateOfConv}
	\sup_{r \in [-d, 0]} \left| F_{Z_d}(r) - \Phi(r) \right| \leq \frac{0.4748 \cdot (2\sqrt{3})^3}{32 \sqrt{d}} = 0.6167... \cdot \frac{1}{\sqrt{d}}.
\end{equation}

\subsection{Different Optimisers}\label{app2}
\paragraph{\bf CMA-ES}
The Covariance Matrix Adaptation Evolution Strategy, also known as CMA-ES is a family of heuristic optimisation techniques for numerical optimisation. It was initially introduced in \cite{Hansen96} (with an extension in \cite{Hansen01}), but modifications have been suggested over the years, leading to many different variants (see for example \cite{vRijn} for some popular modifications). We give here a brief description of the general principle, more details can be found in a tutorial paper by Hansen \cite{HansenTuto} as well as in a more recent presentation by Akimoto and Hansen \cite{AkiHans22}. In our experiments, we use the Diagonal-CMA version, described in \cite{AkiHansDiag}.

The general principle behind CMA-ES is to introduce a covariance matrix associated to a multivariate Gaussian distribution and to learn second order information on this distribution. The mean of this distribution represents the current best guess while the covariance matrix determines the shape of the distribution ellipsoid. At each step, this covariance matrix is used to generate a number of solution candidates. These candidates are evaluated for the function we are trying to optimize and then used to update the mean and covariance matrix of the multivariate Gaussian distribution. The step size for each update is also changed during the algorithm, by comparing the current path length - how much we change the mean of the distribution - with the expected one if steps were independent. If the path length is shorter, the steps are anti-correlated and the step size should be decreased, if the path length is longer it should be increased. To update the mean, a weighted sum of part (or all) of the new samples is done, where weights are higher for better samples and the chosen step size is taken into account. For the covariance, the new covariance matrix will depend both on a weighted sum of the new samples and on the evolution path, the sum of past steps for the mean.

\paragraph{\bf NGOpt:} The second black-box optimiser we used is NGOpt. An older version was presented in \cite{NGOpt8} (in which NGOpt8 is called ABBO), more recent modifications haven't been published but were incorporated in the NGOpt optimiser available in the {\tt{Nevergrad}} Python module. NGOpt is an algorithm selection technique, which will select the best algorithm to run on a problem. It selects the algorithm(s) to run based on features known in advance - dimension of the problem, type of variables, bounds for the variables or the budget for example - as well as based on information obtained during the execution of the algorithms. It can run multiple algorithms in parallel before only continuing with the best one or run multiple algorithms sequentially. CMA-ES is one of the algorithms that NGOpt can select, but given our problem setting (dimension below 10, budget 1000), it does not call the same version of CMA as our CMA experiment. Given the lack of information on the function we are minimizing as well as our lack of budget for hyperparameter tuning and algorithm configuration, using NGOpt's intrinsic algorithm choice is a valuable help in guaranteeing that our chosen optimisers should be good.

In our problem formulation, potential solutions are of the shape $(p_i)_{i \in \{1,\ldots,k\}}$  where for all $i \in\{1,\ldots,N-1\}$ with $i <j$ implies $p_i \leq p_j$. CMA-ES and a large part of the algorithms in NGOpt try to learn some relation between the variables. For this, they use several evaluated points and update the sampling distribution. Most of these optimisers do not work well with constraints (they typically sample points until they fit the constraints, potentially spending a lot of time) therefore our approach was to reorder the solution parameters to always have a sorted point set. After evaluating the fitness of a \emph{sorted} candidate point, the optimiser would consider this to be the fitness of the \emph{non-sorted} original point. The distribution update could then be misled by this sorting.

\paragraph{\bf (1+1)-Evolutionary Strategy:} To limit the impact of the above problem, we also run our experiments with a more traditional (1+1)-Evolutionary Strategy (ES). At each step, the current solution is modified by adding a random variable (in our case a Gaussian random variable) to each element of the solution, scaled by a varying step size. The best solution between the new one and the old solution is then kept for the next step. The step size is modified depending on the mutation success, according to the $1/5$-th rule \cite{SchumSteig}: if too many steps are successful, the step size is increased, and it is decreased otherwise. This iterative process continues until we reach the given budget. There are many different versions of (1+1)-Evolutionary Strategies generally changing the mutation rates. We will be using here the { \tt Parametrized(1+1)} implementation from the {\tt Nevergrad} package with a Gaussian mutation.

\subsection{Additional Tables}\label{app3}
In this last appendix, we place several additional tables for the reader's perusal which give further and exact numerical results relating to the work in Section \ref{sec:num}.

\paragraph{\bf Optimal Point Sets:} Table \ref{points2d} in the main text gives the approximations to the optimal point set $\cS^*(N,2)$ for $3 \leq N \leq 10$ returned by CMA-ES. The tables below show analogous information returned by the two other optimisers as well as the results in dimension 3.

\begin{table}[h!]
    \centering
    \begin{tabular}{|c|c|}
    \hline
        N & Point set obtained by NGOpt  \\
        \hline
        3 & [0.548808, 1.055807] \\
        \hline
        4 & [0.418457, 0.882434, 1.007727]\\
        \hline
        5 & [0.391263, 0.620868, 0.956307, 1.037225] \\
        \hline
        6 & [0.347313, 0.560015, 0.857877, 0.882838, 1.118863]\\
        \hline
        7 & [0.347991, 0.489339, 0.669577, 0.908221, 0.983918, 1.036891]\\
        \hline
        8 & [0.316077, 0.478213, 0.570470, 0.852702, 0.907686, 0.932709 1.067521]\\
        \hline
        9 & [0.285198, 0.470163, 0.568636, 0.610603, 0.871406, 0.926506, 1.012885, 1.141138]\\
        \hline
        10 & [0.272273, 0.443108, 0.532019, 0.631543, 0.819876, 0.864991, 0.925215, 1.008189, 1.132356] \\
        \hline
        15 & [0.19888, 0.3777, 0.444808, 0.543373, 0.546997, 0.575153, 0.63604, \\
        &0.690127, 0.892686, 0.909171, 1.058866, 1.128335, 1.159531, 1.271836] \\
        \hline
        20 & [0.223799, 0.287984, 0.324687, 0.485897, 0.499846, 0.54175, 0.57198, 0.589029, 0.615946, \\
        & 0.7925, 0.817158, 0.858417, 0.908158, 0.938206, 0.960501, 0.982925, 1.101553, 1.159329, 1.319433]\\
        \hline 
    \end{tabular}
    \vspace{1mm}
    \caption{Optimal partitioning points returned by NGOpt in dimension 2.}
    \label{ListNGOPT}
\end{table}

\begin{table}[h!]
    \centering
    \begin{tabular}{|c|c|}
    \hline
        N & Point set obtained by (1+1)-ES  \\
        \hline
        3 & [0.385772, 1.414214] \\
        \hline
        4 & [0.361702, 0.612305, 1.414214]\\
        \hline
        5 & [0.417516, 0.605697, 0.93214, 1.135903] \\
        \hline
        6 & [0.372617, 0.566202, 0.858476, 0.872371, 1.057533]\\
        \hline
        7 & [0.361215, 0.476081, 0.69005, 0.956802, 0.961292, 0.997134]\\
        \hline
        8 & [0.331794, 0.448127, 0.615812, 0.762824, 0.936691, 0.948517, 1.194328]\\
        \hline
        9 & [0.272122, 0.50579, 0.579889, 0.631461, 0.841224, 0.872, 1.029753, 1.189806]\\
        \hline
        10 & [0.300137, 0.411182, 0.477082, 0.713986, 0.775226, 0.810514, 0.878079, 1.002694, 1.203619] \\
        \hline 
        15 & [0.190774, 0.388459, 0.388733, 0.455958, 0.596142, 0.601442, 0.606818, \\
        & 0.828665, 0.83254, 0.873472, 1.012516, 1.022808, 1.121645, 1.174104] \\
        \hline
        20 & [0.249813, 0.325112, 0.359365, 0.461931, 0.483501, 0.574615, 0.5808, 0.584061, 0.606904, \\
        & 0.75478, 0.84873, 0.911128, 0.916237, 0.93133, 0.942195, 0.967771, 1.009936, 1.137893, 1.363866] \\
        \hline
    \end{tabular}
    \vspace{1mm}
    \caption{Optimal partitioning points returned by (1+1)-ES in dimension 2}
    \label{ListES}
\end{table}

\begin{table}[h!]
    \centering
    \begin{tabular}{|c|c|}
    \hline
        N & Point set obtained by NGOpt  \\
        \hline
        3 & [0.795798, 1.286294] \\
        \hline
        4 & [0.640472, 1.117247, 1.310707]\\
        \hline
        5 & [0.566418, 1.047576, 1.083171, 1.150805] \\
        \hline
        6 & [0.489396, 0.98104, 1.013217, 1.048558, 1.087265]\\
        \hline
        7 & [0.448206, 0.878266, 0.91581, 0.982795, 1.063735, 1.125645]\\
        \hline
        8 & [0.42716, 0.800013, 0.863311, 0.910579, 1.006799, 1.11747, 1.242571]\\
        \hline
        9 & [0.408185, 0.77803, 0.830785, 0.882296, 0.890503, 1.011247, 1.169372, 1.43053]\\
        \hline
        10 & [0.399794, 0.603748, 0.89242, 0.913219, 0.926596, 0.926902, 0.967523, 1.16262, 1.287299] \\
        \hline 
        15 & [0.351117, 0.478768, 0.688796, 0.796136, 0.821815, 0.862973, 0.879195, \\
        & 0.90137, 0.9374, 0.946849, 1.045653, 1.094707, 1.186683, 1.571684] \\
        \hline
        20 & [0.34235, 0.389883, 0.57646, 0.680487, 0.762054, 0.793333, 0.822658, 0.827348, 0.886069, \\
        & 0.899673, 0.902049, 0.906803, 0.948087, 0.955605, 1.031147, 1.078856, 1.323314, 1.470775, 1.546017]\\
        \hline
        
    \end{tabular}
    \vspace{1mm}
    \caption{Optimal partitioning points returned by NGOpt in dimension 3}
    \label{ListNGOPT3}
\end{table}

\begin{table}[h!]
    \centering
    \begin{tabular}{|c|c|}
    \hline
        N & Point set obtained by (1+1)-ES  \\
        \hline
        3 & [0.828487, 1.184358] \\
        \hline
        4 & [0.688084, 0.995426, 1.23835]\\
        \hline
        5 & [0.556585, 0.958897, 1.08893, 1.166941] \\
        \hline
        6 & [0.504186, 0.890861, 1.01977, 1.070676, 1.220425]\\
        \hline
        7 & [0.483307, 0.846026, 0.928213, 1.011243, 1.053627, 1.264038]\\
        \hline
        8 & [0.438463, 0.755186, 0.855466, 0.965389, 0.971452, 1.108291, 1.364404]\\
        \hline
        9 & [0.432133, 0.683787, 0.857462, 0.900443, 0.985954, 0.99401, 1.036843, 1.513559]\\
        \hline
        10 & [0.433714, 0.571719, 0.82941, 0.873472, 0.954025, 0.983681, 1.123199, 1.166486, 1.190781] \\
        \hline
        15 & [0.320788, 0.546152, 0.656237, 0.732797, 0.733517, 0.851861, 0.859568, \\
        & 0.958379, 0.997855, 1.003527, 1.044023, 1.112383, 1.3169, 1.538414]\\
        \hline
        20 & [0.308744, 0.45028, 0.612212, 0.618609, 0.652708, 0.769147, 0.791116, 0.831432, 0.852191, \\
        & 0.877652, 0.927507, 0.950297, 0.997164, 1.051177, 1.070492, 1.172144, 1.212012, 1.531995, 1.640873] \\
        \hline 
    \end{tabular}
    \vspace{1mm}
    \caption{Optimal partitioning points returned by (1+1)-ES in dimension 3}
    \label{ListES3}
\end{table}

\begin{table}[h!]
    \centering
    \begin{tabular}{|c|c|}
    \hline
        N & Point set obtained by CMA-ES  \\
        \hline
        3 & [0.852603, 1.229802] \\
        \hline
        4 & [0.649511, 1.103493, 1.233462]\\
        \hline
        5 & [0.557872, 0.993939, 1.111959, 1.191173] \\
        \hline
        6 & [0.47672, 0.960814, 1.033193, 1.088396, 1.110393]\\
        \hline
        7 & [0.450648, 0.878901, 0.905647, 1.027906, 1.096725, 1.127983]\\
        \hline
        8 & [0.417508, 0.804702, 0.962827, 0.972907, 0.985132, 0.995874, 1.222424]\\
        \hline
        9 & [0.385481, 0.824133, 0.863861, 0.864913, 0.950659, 0.968246, 1.027493, 1.284941]\\
        \hline
        10 & [0.386811, 0.644514, 0.887476, 0.897878, 0.900528, 0.936027, 0.949786, 1.220786, 1.309354] \\
        \hline 
        14 & [0.358027, 0.497384, 0.609137, 0.81638, 0.875619, 0.901655, 0.904081, \\
        & 0.95021, 0.959341, 0.975178, 1.012549, 1.015003, 1.248424, 1.633825]\\
        \hline
        19 & [0.283066, 0.444381, 0.588593, 0.67853, 0.741726, 0.756299, 0.785411, 0.793158, 0.891821, \\ 
        & 0.901059, 0.918331, 0.928474, 0.947623, 1.036847, 1.039228, 1.125772, 1.135725, 1.294058, 1.436106] \\
        \hline
    \end{tabular}
    \vspace{1mm}
    \caption{Optimal partitioning points returned by CMA-ES in dimension 3}
    \label{ListCM3}
\end{table}

\paragraph{\bf Exact Discrepancy Values:} Figure \ref{fig:percentagecomparison} in the main text gives a percentage comparison of the expected $\mathcal{L}_2$-discrepancies of the various point sets. For completeness sake and as a final attachment, Table \ref{Compare2} gives the exact expected discrepancy of the point same point sets.

\begin{table}[h!]
    \centering
    \begin{tabular}{|c|c|c|c|c|c|}
    \hline
        $N$ & $R(N,2)$ & $\cS(N,2)$ & CMA-ES & NGOpt & (1+1)-ES  \\
         \hline
         3 & 0.04614 & 0.02917 & 0.02696 & 0.02682 & 0.03061\\
         4 & 0.03441 & 0.02035 & 0.01883 & 0.01891 & 0.023\\
         5 & 0.02774 & 0.01560 & 0.01453 & 0.01464 & 0.01472 \\
         6 & 0.02327 & 0.01272 & 0.01191 & 0.01195 & 0.01192 \\
         7 & 0.01990 & 0.01070 & 0.01007 & 0.01006 & 0.0108 \\
         8 & 0.01714 & 0.009202 & 0.008722 & 0.008668 & 0.008866\\
         9 & 0.01547 & 0.008138 & 0.007738 &  0.007764 & 0.007829\\
         10 & 0.01492 & 0.007266 & 0.006961 & 0.006939 & 0.007051 \\
         15 & 0.009320 & 0.004746 & 0.004668 & 0.004668 & 0.004647 \\
         20 & 0.006955 & 0.003532 & 0.003541 & 0.003558 & 0.003478 \\
         \hline
    \end{tabular}
	\vspace{3mm}
    \caption{Comparison of the expected $\mathcal{L}_2$-discrepancies of different point sets: uniformly distributed random points $R(N,2)$, the equivolume partition corresponding to $\cS(N,2)$ and the best sets obtained by the three black-box optimisers. The discrepancy values are evaluated with 10000 repetitions.}
    \label{Compare2}
\end{table}

\begin{table}[h!]
    \centering
    \begin{tabular}{|c|c|c|c|c|c|}
    \hline
        $N$ & $R(N,3)$ & $\cS(N,3)$ & CMA-ES & NGOpt & (1+1)-ES  \\
         \hline
         3 & 0.0295 & 0.0220 & 0.01960 & 0.01968 & 0.01979\\
         4 & 0.0221 & 0.0156 & 0.01433 & 0.01431 & 0.01433\\
         5 & 0.0176 & 0.0121 & 0.01107 & 0.01104 & 0.01113 \\
         6 & 0.0148 & 0.00997 & 0.008983 & 0.008979 & 0.009115 \\
         7 & 0.0125 & 0.00836 & 0.007634 & 0.007718 & 0.007720 \\
         8 & 0.011 & 0.00727 & 0.0067 & 0.006712 & 0.006741\\
         9 & 0.00989 & 0.00640 & 0.005973 &  0.005949 & 0.005967\\
         10 & 0.00876 & 0.005713 & 0.005352 & 0.005358 & 0.005345 \\
         15 & 0.00594 & 0.00376 & 0.003569 & 0.003561 & 0.003579 \\
         20 & 0.00442 & 0.0028 & 0.002707 & 0.002692 & 0.002708 \\
         \hline
    \end{tabular}
	\vspace{3mm}
    \caption{Comparison of the expected $\mathcal{L}_2$-discrepancies of different point sets: uniformly distributed random points $R(N,3)$, the equivolume partition corresponding to $\cS(N,3)$ and the best sets obtained by the three black-box optimisers. The discrepancy values are evaluated with 10000 repetitions.}
    \label{Compare3}
\end{table}

\end{document}